\documentclass[11pt,reqno]{amsart}

\usepackage[utf8]{inputenc} % acentos
\usepackage[T1]{fontenc}
\usepackage{lmodern}

\usepackage{etex}
\usepackage{geometry}
\usepackage{graphicx}
\usepackage{tikz}
\usepackage{tikz-cd}
\usepackage{amssymb}
\usepackage{url}
\usepackage{calligra}
\usepackage{stackrel}
\usepackage{wrapfig}
\usepackage{caption}
\usepackage{subfig}
\usepackage{amsmath}

\usepackage[retainorgcmds]{IEEEtrantools}
    
\usepackage{mathtools}
\usepackage{pgfplots} % Tikz, para fazer gr??ficos
\usepackage{tikz-cd}
\usepackage{empheq}
\usepackage{MnSymbol}

\usepackage{float} % se quiser usar o H para as imagens ficarem Here
\usepackage{enumitem} % enumera????o
\usepackage{color} % cores

\usepackage{hyperref}
\usepackage{bookmark}

\usepackage[backend=bibtex,sorting=nyt,style=alphabetic,
    firstinits,isbn=false,maxnames=4,maxalphanames=4,minalphanames=4]{biblatex}
\AtEveryBibitem{%
  \clearlist{language}%
  \clearlist{Language}%
  \clearfield{note}%
}

\DeclareMathAlphabet{\mathscr}{T1}{calligra}{m}{n}
\usepackage{amsthm}

\theoremstyle{plain}
\newtheorem{theorem}{Theorem}[section]
\newtheorem{corollary}[theorem]{Corollary}
\newtheorem{lemma}[theorem]{Lemma}
\newtheorem{proposition}[theorem]{Proposition}

\theoremstyle{definition}
\newtheorem{definition}[theorem]{Definition}

\newtheorem{remark}[theorem]{Remark}
\newtheorem*{notation}{Notation}

\DeclareMathOperator{\image}{im}

\DeclareMathOperator{\R}{{\mathbb{R}}}
\DeclareMathOperator{\C}{{\mathbb{C}}}

\newcommand{\id}{\mathrm{id}}
\newcommand{\Ad}{\mathrm{Ad}}
\newcommand{\ad}{\mathrm{ad}}
\newcommand{\g}{{\mathfrak{g}}}

\newcommand{\ttt}{{\mathfrak{t}}}

\DeclareMathOperator{\spn}{span}
\DeclareMathOperator{\im}{Im}

\allowdisplaybreaks

\title{Partial coherent state transforms, $G\times T$-invariant K\"ahler structures and geometric quantization of cotangent bundles of compact Lie groups}

\author{Jos\'e~M.~Mour\~{a}o}
\email{jmourao@tecnico.ulisboa.pt}
\author{Jo\~{a}o~P.~Nunes}
\email{jpnunes@math.tecnico.ulisboa.pt}
\address{Departament of Mathematics
and Center for Mathematical Analysis, Geometry and Dynamical Systems\\
Instituto Superior T\'ecnico\\
Av. Rovisco Pais\\
1049-001 Lisboa\\
Portugal}
\author{Miguel~B.~Pereira}
\email{miguel.barbosa@math.uni-augsburg.de}
\address{Institut f\"ur Mathematik\\
University of Augsburg\\
 86159 Augsburg\\ Germany}

\date{\today}

\begin{document}

\begin{abstract}In this paper, we study the analytic continuation to complex time of the Hamiltonian flow of certain $G\times T$-invariant 
functions on the cotangent bundle of a compact connected Lie group $G$ with maximal torus $T$. Namely, we will take the Hamiltonian flows of one 
$G\times G$-invariant function, $h$, and one $G\times T$-invariant function, $f$.
Acting with these complex time Hamiltonian flows on $G\times G$-invariant K\"ahler structures gives new $G\times T$-invariant, but not $G\times G$-invariant, K\"ahler structures on $T^*G$. We study the Hilbert spaces ${\mathcal H}_{\tau,\sigma}$ corresponding to the quantization of $T^*G$ with respect to these non-invariant 
K\"ahler structures. On the other hand, by taking the vertical Schr\"odinger polarization as a starting point, the above 
$G\times T$-invariant Hamiltonian flows also generate families of mixed polarizations $\mathcal{P}_{0,\sigma}, \sigma \in {\mathbb C}, \im \sigma >0$. Each of these mixed polarizations is globally given by a direct sum  of an integrable real distribution and of a complex distribution that defines a K\"ahler structure on the leaves of a foliation of $T^*G$.  
The geometric quantization of $T^*G$ with respect to these mixed polarizations gives rise to unitary partial coherent state transforms, corresponding to KSH maps as defined in \cite{KMN1,KMN2}.   
\end{abstract}

\maketitle
\tableofcontents
\section{Introduction}

Geometric quantization is an approach to the mathematical problem of quantization which aims at defining the quantization of a symplectic manifold $(M,\omega)$ which includes, in particular, the assignement of a Hilbert space of quantum states to $(M,\omega)$. This assignement is far from unique as it depends on the choice of an additional structure, a polarization, which is an involutive Lagrangian distribution in $TM\otimes {\C}$. The dependence of quantization on this choice is one of the most important objects of study in geometric quantization. 

Among symplectic manifolds, cotangent bundles of compact Lie groups, $T^*G$, provide a rich class of spaces for the study of interactions between analysis and representation theory, 
K\"ahler geometry  and geometric quantization. On one hand, Hall's generalization of the classical coherent state transform of Segal-Bargmann \cite{Ha1}, which depends essentially on properties of the heat kernel on $G$ and on its complexification $G_{\C}$, corresponds to a natural pairing map between 
the quantizations of $T^*G$ in the vertical (or Schr\"odinger) polarization and in the K\"ahler polarization provided by the identification $T^*G \cong G_{\C}$ 
\cite{Ha2}. In fact, these two polarizations can be connected by a continuous family of $G\times G$-invariant K\"ahler polarizations, which are related among themselves by compositions of Hall's coherent state transforms (CST) \cite{FMMN1,FMMN2,KW}. On the other hand, these, as well as other more general 
\cite{N,KMN1,KMN2}, natural families of $G\times G$-invariant K\"ahler structures are also very interesting from the point of view of K\"ahler geometry. Indeed, they are generated by the analytic continuation to complex time of Hamiltonian flows on $T^*G$, of a so-called complexifier Hamiltonian function \cite{Th,HK}, and correspond to geodesics for the Mabuchi affine connection on the space of K\"ahler metrics on $T^*G$ \cite{KMN1,MN}. In \cite{BHKMN}, similar families of quantizations for more general symmetric spaces of compact type are presented and studied in the infinite geodesic time limit.

In this paper, we extend these results by considering K\"ahler structures which are not $G\times G$-invariant. For that purpose, we consider $G\times T$-invariant Hamiltonian flows analytically continued to complex time, where $T\subset G$ is a maximal torus. Acting on the $G\times G$-invariant K\"ahler structures of \cite{KMN1} we give examples of new, $G\times T$- but not $G\times G$-invariant, K\"ahler structures on $T^*G$. We study the geometric quantization of $T^*G$ with respect to these K\"ahler polarizations. 

Acting with the $G\times T$-invariant quadratic complexifiers on the (real) Schr\"odinger polarization gives interesting mixed polarizations which define foliations of $T^*G$ by K\"ahler submanifolds isomorphic to $T^*T\cong T_{\C}$. We also study the quantization of $T^*G$ with respect to these mixed polarizations and show that 
they are related to (unitary) partial coherent state transforms on $G$ which are ``partially holomorphic'' analogs of Hall's CST.

\begin{remark}While in this paper we have considered  $G\times T$-invariant Hamiltonian flows generated by complexifiers which correspond to a strictly convex 
$\Ad$-invariant function on the Lie algebra of $G$ and a strictly convex quadratic function on the Lie algebra of $T$, $\frak t$, we expect most of the results to generalize straightforwardly to the case when the quadratic function on 
$\frak t$ is replaced by a more general strictly convex function. 
\end{remark}

\section{Geometry of $T^*G$}

\label{prel}

\subsection{Preliminaries}
Let $G$ be a compact connected Lie group of dimension $n$ and rank $r$. We assume that its Lie algebra, $\g$, is equipped with an $\Ad$-invariant inner-product 
$\langle\cdot, \cdot\rangle$. We will also assume the usual identifications given by left-translation
\begin{equation*}
    T^*G \cong G\times \g^* \cong G\times \g,
\end{equation*}
where $\g$ and $\g^*$ get identified by means of $\langle\cdot,\cdot\rangle.$ 
With these identifications, we have for the tangent spaces
\begin{equation}\label{tangentspace}
T_{(x,y)}(T^*G) \cong \g \oplus \g,\quad (x,y)\in G\times \g.
\end{equation}

Recall also that, by the polar decomposition,
\begin{eqnarray*}\nonumber
T^*G = G\times {\g} & \stackrel{\cong}{\to} & G_{\C} \\ \nonumber
(x,y) & \mapsto & xe^{iy} ,
\end{eqnarray*}
where $G_{\C}$ is the complexification of $G$.

The standard $G\times G$-action (where we take a left-action for the first factor and a right-action for the second) on $T^*G$ then corresponds to
\begin{equation*}
    g\cdot (x,y) \cdot h = (gxh, \Ad_{h^{-1}}(y)), \quad g,h\in G, \quad (x,y)\in G\times \g = T^*G.
\end{equation*}

Let $\{T_j\}_{j=1,\dots,n}$ be an orthonormal basis of $\g$ and let 
$\{X_j\}_{j=1,\dots,n}$ be the corresponding basis of left-invariant vector fields on $G$. 
We will also denote by $\{X_j\}_{j=1,\dots,n}$ the corresponding left-invariant vector fields on  
$T^*G=G\times \g$ with zero component along the second summand in (\ref{tangentspace}). Let $\{w^j\}_{j=1,\dots, n}$ be the corresponding dual basis of left-invariant one-forms 
on $G$ and denote by the same symbols their pull-backs to $T^*G$ by the canonical projection $T^*G\to G.$
Let $\{y^j\}_{j=1,\dots,n}$ be Cartesian coordinates on $\g$ associated to the above orthonormal basis.

Recall that $T^*G$ has a canonical symplectic structure $\omega = -d\theta$, where $\theta$ is the canonical one-form. In the above
coordinates, we have
\begin{equation*}
    \theta = \sum_{j=1}^n y^j w^j
\end{equation*}
and
\begin{equation*}
    \omega = \sum_{j=1}^n \bigg(w^j\wedge dy^j + \frac12 \sum_{k,l=1}^n c_{kl}^j y^j w^k\wedge w^l\bigg),
\end{equation*}
where $\{c_{kl}^j\}_{j,k,l=1,\dots n}$ are the (totally anti-symmetric) structure constants of $\g$ relative to the above orthonormal basis.
In terms of the decomposition in (\ref{tangentspace}) we have
\begin{equation*}
    \theta \colon (G \times \g) \times (\g \oplus \g) \longrightarrow \mathbb{R},
\end{equation*}
with
	\begin{equation}\label{theta}
		\theta_{(x,y)}\left( \left(
		\begin{array}{c}
		U\\
		V\\
		\end{array}
		\right) \right) = \langle y, U \rangle, \quad U, V\in \frak g
	\end{equation}
	while $\omega \colon (G \times \g) \times (\g \oplus \g) \times (\g \oplus \g) \longrightarrow \mathbb{R}$ is given by
	\begin{IEEEeqnarray}{rCl}
		 \nonumber\omega_{(x,y)} \left( \left(\begin{array}{c}
			U\\
			V\\
		\end{array}
		\right), \left( \begin{array}{c}
			W\\
			Z\\
		\end{array}
		\right) \right) & = & \left(\begin{array}{cc}
			U & V\\
		\end{array}
		\right)
		\left( \begin{array}{cc}
			-\ad_y & \id \\
			-\id & 0\\
		\end{array}
		\right)
		\left( \begin{array}{c}
			W\\
			Z\\
		\end{array}
		\right)\\
		& = & \label{eq:omega2} \langle U, Z \rangle - \langle V, W \rangle + \langle y, [U,W] \rangle,\quad U, V, W, Z \in \frak g.
	\end{IEEEeqnarray} 

A left-invariant function $g\colon T^*G\cong G\times \g \to \R$ is determined by a function on $\g$ which we will denote by the same symbol
$g\colon \g\to \R.$ We will denote by $u_g$ the gradient of $g$, that is
$$
\langle u_g(y),A\rangle = dg_y (A), \quad y,A\in \g.
$$
$H_g$ will denote the Hessian of $g$.
Recall that if $g\colon T^*G\to \R$ is $G\times G$-invariant, so that the associated function $g\colon \g\to \R$ is $\Ad$-invariant, then (see \cite{KMN1}, Lemma 3.4)
\begin{equation}\label{commutes}
[y,u_g(y)]=0, \quad {\rm and}\quad \ad_{u_g(y)}=\ad_y\, H_g(y) = H_g(y)\, \ad_y,\quad y\in \g.
\end{equation}
One also has, if $g$ is $G\times G$-invariant, that
\begin{equation}\label{adinvu}
\Ad_x u(y) = u(\Ad_x y),\quad x\in G, y\in \g,
\end{equation}
and for the Hessian, as a linear map ${\frak g} \to {\frak g}$, one obtains in that case
\begin{equation}\label{adhessian}
H_g(\Ad_x y) = \Ad_x \circ H_g(y) \circ \Ad_{x^{-1}},\quad x\in G.
\end{equation}

\begin{proposition}\label{hvf}
	Let $g \colon T^*G \longrightarrow \mathbb{R}$ be a left-invariant function. Then, its Hamiltonian vector field $X_g$ is given by:
	\begin{equation}
	\label{eq:hamvfg}
		X_g |_{(x,y)} = \big( u_g(y), [y,u_g(y)] \big).
	\end{equation}
	\end{proposition}
\begin{proof}
	Using equation (\ref{eq:omega2}), we obtain
	\begin{equation}\nonumber
		\omega_{(x,y)} \left( \left(\begin{array}{c}
		u_g (y) \\
		\left[ y,u_g(y) \right] \\
		\end{array}
		\right), \left( \begin{array}{c}
		W\\
		Z\\
		\end{array}
		\right) \right) = \langle u_g(y), Z \rangle.
	\end{equation}
	From the definition of gradient,
	\begin{equation}\nonumber
		dg|_{(x,y)} (Z) = \langle u_g (y), Z \rangle.
	\end{equation}
	This proves equation (\ref{eq:hamvfg}).
\end{proof}

Let now $$h\colon T^*G\to \R$$ be an Hamiltonian function such that
\begin{enumerate}[label={\roman*)}]
\item $h$ is $G\times G$- invariant. This implies that $h$ is determined by an $\Ad$-invariant function on $\g$ which we also denote by $h$;
\item The Hessian $H_h$ on $\g$ is positive definite everywhere;
\item The operator norms $\vert\vert H_h(y)\vert\vert$, $y\in \g$, have a positive lower bound on $\g$.\footnote{This condition is not strictly necessary for 
K\"ahlericity but we will assume it for simplicity. (See Lemma 3.1 in \cite{KMN1}).}
\end{enumerate}

The Hamiltonian vector field of $h$ is given by
\begin{equation*}
%\label{eq:hamvfh} 
X_h |_{(x,y)}  =  \big( u_h(y), 0 \big),
\end{equation*}
and the corresponding Hamiltonian flow is
\begin{equation}
\label{eq:flowofXh} 
\phi^t_{X_h}(x,y)  =  \left( x e^{t u_h(y)} , y \right), 
\end{equation}
for $(x,y)\in G\times \g\cong T^*G,\, t\in \R$. Recall now that the analytic continuation of this Hamiltonian flow 
to imaginary time gives $G\times G$-invariant K\"ahler structures on $T^*G$, as follows. (See \cite{KMN1} and also \cite{N}.)
Let $\tau = \tau_1+i\tau_2\in \C, \tau_1,\tau_2\in \R$ and let
\begin{equation*}
    \C^+=\{\tau \in \C \ | \ {\rm Im}\, \tau >0\}.
\end{equation*}
For $\tau\in C$ consider the maps
\begin{equation}
\begin{array}{ccccc}\label{psitau}
T^{*}G \cong G\times {\g} &{\stackrel{\alpha_h}{\to}} & G\times {\g} &{\stackrel{\psi_{\tau}}{\to}}&  G_{\C} \\
(x,y) &\mapsto & (x,u_h (y)) &\mapsto&  xe^{\tau u_h (y)}.
\end{array}
\end{equation}
Note that $\psi_\tau\circ \alpha_h$ is a diffeomorphism if ${\rm Im}\,\tau \neq 0.$

\begin{proposition}\label{cstau}{\emph{\cite{KMN1}}} For $\tau\in \C^+$, let $J_{\tau,0}$ be the the pull-back of the canonical complex structure on $G_{\C}$ by $\psi_\tau \circ \alpha_h.$ Then, $(T^*G,\omega, J_{\tau,0})$ is a K\"ahler manifold. A (global) K\"ahler potential is given by the Legendre transform of $h$,
$$
\kappa_{\tau,0}(x,y) = 2\tau_2 \left(\langle y,u_h(y)\rangle - h(y)\right).
$$
\end{proposition}

\subsection{$G\times T$-invariant quadratic Hamiltonians}
\label{sectionquadratic}

Let $\ttt \subset \frak g$ be a Cartan subalgebra corresponding to a maximal torus $T\subset G$. Recall that $\Ad$-invariance of $\langle\cdot,\cdot\rangle$ implies that if $A\in \ttt$ then $\ad_A\colon \g \to \ttt^\perp$.

Let $$f\colon T^*G\cong G\times \g \to \R$$ be the $G\times T$-invariant Hamiltonian function determined by a real-valued function on $\g$, which we also denote by $f$, 
given by the symmetric form 
$$
f(y) = \frac12 \langle y, Fy\rangle,
$$
where $F\colon \g\to \g$ is a linear real self-adjoint map on $\g$ satisfying
\begin{enumerate}[label={(\roman*)}]
\item $\image F \subset \ttt$;
\item $F_{\vert_{\ttt^\perp}} =0$;
\item $F_{\vert_{\ttt}} >0$,
\end{enumerate}
so that $f$ is determined by a positive quadratic form on the Cartan subalgebra $\ttt$. For simplicity, we will henceforth denote by the same symbol, $F$, 
both the linear map $F$ and its restriction $F_{\vert_{\frak t}}$. No confusion should arise from the context and, in particular, $\det F$ will always stand for 
$\det F_{\vert_{\frak t}}>0$. 

\begin{lemma}The Hamiltonian vector field of $f$ is given by
\begin{equation*}
%\label{eq:hamvff} 
X_f |_{(x,y)}  =  \big( Fy, [y,Fy] \big)
\end{equation*}
and the corresponding Hamiltonian flow is
\begin{equation*}
%\label{eq:flowofXf} 
\phi^s_{X_f}(x,y)  =  \left( xe^{s F y}, e^{-s \, \ad_{Fy}} y\right),
\end{equation*}
for $(x,y)\in G\times \g \cong T^*G,\, s\in \R.$
\end{lemma}

\begin{proof}The expression for $X_f$ follows from Proposition \ref{hvf}. On the other hand,
let $(x(s),y(s))$ be the integral curve of $X_f$ that on $s=0$ goes through $(x_0,y_0)$. This corresponds to the initial value problem
	\begin{equation*}
		\left\{ 
			\begin{IEEEeqnarraybox}[\IEEEeqnarraystrutmode][c]{rCl}
				\dot{x} & = & dL_x Fy \\
				\dot{y} & = & [y,Fy] \\
				x(0) & = & x_0 \\
				y(0) & = & y_0
			\end{IEEEeqnarraybox}.
		\right.
	\end{equation*}
	Since $F y_0 \in \frak t$, $\ad_{Fy_0}$ maps to ${\frak t}^\perp$. Therefore, $F \circ \ad_{Fy_0} = 0$. From this we conclude that $F e^{-s \, \ad_{Fy_0}} y_0 = F y_0$. We use this fact to prove that $y(s) = e^{-s \, \ad_{Fy_0}} y_0$:
	\begin{IEEEeqnarray}{rCl}\nonumber\frac{d}{ds} e^{-s \, \ad_{Fy_0}} y_0 & = & - \ad_{Fy_0} e^{-s \, \ad_{Fy_0}} y_0 \\ \nonumber
		& = & \left( e^{-s \, \ad_{Fy_0}} y_0, Fy_0 \right) \\ \nonumber
		& = & \left( e^{-s \, \ad_{Fy_0}} y_0, F e^{-s \, \ad_{Fy_0}} y_0 \right).
	\end{IEEEeqnarray}
	$x$ must satisfy $\dot{x} = dL_x Fy = dL_x Fy_0$, $x(0) = x_0$. Therefore, $x(s) = x_0 e^{s F y_0}$.
\end{proof}

Recall now the Hamiltonian flow (\ref{eq:flowofXh}) of the $G\times G$-invariant Hamiltonians $h$ in Proposition \ref{cstau}.

\begin{proposition}
The Hamiltonian flows of $h$ and $f$ commute, namely
	\begin{IEEEeqnarray}{rCl}\nonumber
		\phi^t_{X_h} \circ \phi^s_{X_f}(x,y) & = & \phi^s_{X_f} \circ \phi^t_{X_h}(x,y) \\ \nonumber
		& = & \left( x e^{t u_h(y)} e^{s F y}, e^{-s \, \ad_{Fy}} y\right),
	\end{IEEEeqnarray}
	for $t,s\in \R$.
\end{proposition}

\begin{proof}
This can checked by verifying that $[X_h, X_f] = 0$ or by direct computation. Indeed, 
	\begin{IEEEeqnarray}{rCl}\nonumber
		\phi^s_{X_f} \circ \phi^t_{X_h}(x,y) & = & \phi^s_{X_f} \left( x e^{t u_h(y)} , y \right) \\ \nonumber
		& = & \left( x e^{t u_h(y)} e^{s F y}, e^{-s \, \ad_{Fy}} y\right).
	\end{IEEEeqnarray}
	implies that $\phi^t_{X_h} \circ \phi^s_{X_f} = \phi^s_{X_f} \circ \phi^t_{X_h}$, for $t,s \in \R$. In fact, equation (\ref{adinvu}) implies that
    \begin{equation*}
        e^{s Fy} e^{t u_h( e^{-s \, \ad_{Fy}} y )} = e^{t u_h(y)} e^{s F y}. \qedhere
    \end{equation*}
\end{proof}

\begin{lemma}\label{derivatives}
	For $s,t \in {\R}$, the tangent maps $D\phi^t_{X_h} ,D\phi^s_{X_f}\colon \g \oplus \g \longrightarrow \g \oplus \g$,  are given at $(x,y)\in T^*G$ by
	\begin{IEEEeqnarray}{rCl}
		\label{eq:dflowXh} D\phi^t_{X_h}& = & \left( \begin{array}{cc}
			e^{-t \ad_{u_h(y)}} & \frac{\id - e^{-t \ad_{u_h(y)}}}{\ad_{u_h(y)}} H_h(y) \\
			0 & \id
		\end{array} \right), \\
		\label{eq:dflowXf} D\phi^s_{X_f} & = & \left( \begin{array}{cc}
			e^{-s \, \ad_{Fy}} & sF \\
			0 & e^{-s \, \ad_{Fy}}(s \, \ad_y\circ F + \id)
		\end{array} \right).
	\end{IEEEeqnarray}
\end{lemma}

\begin{proof}
	 To prove (\ref{eq:dflowXh}), let $\gamma(s) = (\gamma^1(s), \gamma^2(s)) = (xe^{sU}, y+sV)$, $U,V\in \frak g$. Then,
	\begin{IEEEeqnarray}{rCl}\nonumber
		D_{(x,y)} \phi^t_{X_h}(U,V) & = & \frac{d}{ds} \Bigg|_{s=0} \phi^t_{X_h}(\gamma^1(s), \gamma^2(s)) \\ \nonumber
		& = & \frac{d}{ds} \Bigg|_{s=0} \phi^t_{X_h}(xe^{sU}, y+sV) \\ \nonumber
		& = & \frac{d}{ds} \Bigg|_{s=0} \left( x e^{sU} e^{t u_h(y+sV)}, y + sV\right) \\ \nonumber
		& = & \left( dL_{xe^{tu_h(y)}}\left( \Ad_{e^{-tu_h(y)}} U + \frac{\id - e^{-t \ad_{u_h(y)}}}{\ad_{u_h(y)}} H_h(y) V \right), V \right) \\ \nonumber
		& = & \left( dL_{xe^{tu_h(y)}}\left( e^{-t\ad_{u_h(y)}} U + \frac{\id - e^{-t \ad_{u_h(y)}}}{\ad_{u_h(y)}} H_h(y) V \right), V \right).
	\end{IEEEeqnarray}
	To prove (\ref{eq:dflowXf}), by the same reasoning as before,
	\begin{IEEEeqnarray*}{rCl+x*}
        \IEEEeqnarraymulticol{3}{l}{D_{(x,y)} \phi^s_{X_f}(U,V)}\\
        \quad & = & \frac{d}{dt} \Bigg|_{t=0} \phi^s_{X_f}(xe^{tU},y+tV) \\ 
		\quad & = & \frac{d}{dt} \Bigg|_{t=0} \left( xe^{tU} e^{sF(y+tV)}, e^{-s \, \ad_{F(y+tV)}}(y+tV)\right) \\ 
		\quad & = & \Bigg( dL_{xe^{sFy}} \left( e^{-s \, \ad_{Fy}} U + \frac{\id - e^{-s \, \ad_{Fy}}}{\ad_{Fy}} FV \right), e^{-s \, \ad_{Fy}}(-s \, \ad_{FV})(y) + e^{-s \, \ad_{Fy}} V \Bigg)  \\ 
		\quad & = & \left( dL_{xe^{sFy}} \left( e^{-s \, \ad_{Fy}} U + sFV \right), e^{-s\,\ad_{Fy}}(s \, \ad_y\circ F + \id) V \right). & \hfill\qedhere
	\end{IEEEeqnarray*}
\end{proof}

\begin{lemma}\label{composederivatives}Let $t,s \in {\R}$. Then,
\begin{equation*}
%\label{eqcompositionderivative}
D\left(\phi_{X_h}^{-t} \circ \phi_{X_f}^{-s}\right)_{(\phi_{X_h}^{t} \circ \phi_{X_f}^{s})(x,y)}=
\left[\begin{array}{cc}
			e^{t \ad_{u_h(y)}} e^{s \, \ad_{Fy}} & \frac{\id - e^{t \ad_{u_h(y)}}}{\ad_{u_h(y)}} H_h(y) e^{s \, \ad_{Fy}} -s F \\
			0 & -s \, \ad_y \circ F + e^{s \, \ad_{Fy}}
		\end{array}\right].
\end{equation*}
\end{lemma}

\begin{proof}This follows from the chain rule, Lemma \ref{derivatives} and equations (\ref{commutes}), (\ref{adinvu}).
\end{proof}

The following will also be useful later on.

\begin{lemma}\label{invertible}
For every $y \in \g, s\in {\R}$, the linear map $e^{s \, \ad_{Fy}} - s \, \ad_y \circ F$ is an automorphism of $\g$.
\end{lemma}
\begin{proof} Let $s\neq 0$. 
	We prove that $\ker\left( e^{s \, \ad_{Fy}} - s \, \ad_y \circ F \right) = 0$. Let $V \in \ker\left( e^{s \, \ad_{Fy}} - s \, \ad_y \circ F \right)$. Then,
	\begin{equation*}
		e^{s \, \ad_{Fy}}V - s \, \ad_y \circ F\,V = 0.
	\end{equation*}
	We can split the terms of this equation that belong to ${\frak t}$ and those that belong to ${\frak t}^{\perp}$:
	\begin{equation}\nonumber
		\underbrace{V^{\parallel}}_{\in {\frak t}} + \underbrace{V^\perp}_{\in {\frak t}^\perp} + \sum_{k=1}^{\infty} \frac{s^k}{k!} \underbrace{ \ad^k_{\underbrace{Fy}_{\in {\frak t}}} V}_{\in {\frak t}^\perp} - s \underbrace{\ad_y \circ\underbrace{F\,V}_{\in {\frak t}}}_{\in {\frak t}^\perp} = 0,
	\end{equation}
	from which we conclude that $V^\parallel = 0$. Therefore,
	\begin{equation*}
		V^\parallel = 0 \Longrightarrow  \ad_y \circ F V = 0 \Longrightarrow e^{s \, \ad_{Fy}} V = 0 \Longrightarrow V = 0. \qedhere
	\end{equation*}
\end{proof}

\begin{definition}Let $F$ be a linear self-adjoint map on $\frak g$ satisfying the properties listed in the beginning of this Section and let $h$ be as in Proposition 
\ref{cstau}. For $\tau,\sigma\in \C$ define
\begin{equation*}
\begin{array}{rcl}
A_{\tau,\sigma}\colon  T^*G &\to& G_{\C}\\
(x,y)&\mapsto& xe^{\tau u_h(y)} e^{\sigma Fy}.
\end{array}
\end{equation*}
\end{definition}

% Let 
% \begin{equation}\nonumber
% \begin{array}{rcl}
% P\colon  T^*G &\to& G_{\C}\\
% (x,y)&\mapsto& xe^{iy},
% \end{array}
% \end{equation}
% be the inverse of polar decomposition map, which is a global diffeomorphism. 
% For $\tau=\tau_1+i\tau_2\in \C, \tau_1,\tau_2\in \R, \sigma\in \R$, we can then also write
% $$
% A_{\tau,\sigma} = P \circ \alpha_{\tau_2 h} \circ \phi^{\sigma}_{X_f} \circ \phi^{\tau_1}_{X_h}.
% $$
Note that, $A_{\tau,0}=\psi_\tau \circ \alpha_h$ in (\ref{psitau}) and that, for $\tau\in \C, \sigma\in \R$, also
\begin{equation*}
%\label{decompflow}
A_{\tau,\sigma} = \psi_\tau \circ \alpha_h \circ \phi^\sigma_{X_f}.
\end{equation*}

This implies the following 
\begin{lemma}For $\tau\in \C^+, \sigma\in\R,$ the map $A_{\tau,\sigma}$ is a global diffeomorphism.
\end{lemma}

Actually, a much stronger result holds.

\begin{theorem}\label{sure}Let $\tau, \sigma\in \C^+$. Then $A_{\tau,\sigma}$ is a global diffeomorphism.
\end{theorem}
\begin{proof}
Let $r$ be the rank of $G$ and let $T_{\C}\cong ({\C}^*)^r$ be the complexification of $T$.
Consider the standard right $T_{\C}$-action on $G_{\C}$.
>From Theorems 1.12 and 1.23 in \cite{S}, at any point in $G_{\C}$ there are holomorphic slices for the action of $T_{\C}$. That is, for each $p\in G_{\C}$ 
there is a locally closed analytic subspace $S\subset G_{\C}$ and a $T_{\C}$-equivariant map
$$
\phi \colon S\times T_{\C} {\rightarrow} G_{\C}
$$
which is a biholomorphism onto a $T_{\C}$-invariant open neighbourhood of $p$ in $G_{\C}$. From (\ref{psitau}) we have also a diffeomorphism
$$
\psi_\tau \circ \alpha_h \colon T^*G {\rightarrow} G_{\C}.
$$
Let $\beta = (\psi_\tau \circ \alpha_h)^{-1}\circ \phi$ and let $\tilde \omega = \beta^*\omega$ be pull-back the standard symplectic structure on $T^*G$ to $S\times T_{\C}$ which, by Proposition \ref{cstau} becomes a K\"ahler manifold. 
The standard right action of $T$ on $T^*G$,
$$
t\cdot (x,y) = (xt, \Ad_{t^{-1}}y), \,\,\,\, t\in T,
$$
is Hamiltonian with moment map $\mu(x,y) = y^{\frak t}$, where $y^{\frak t}$ is the component of $y\in \frak g$ in the decomposition $\frak g = \frak t \oplus \frak t^{\perp}$. The map $\psi_\tau\circ \alpha_h$, by $\Ad$-invariance of $h$, intertwines this action with the standard right action of $T$ on $G_{\C}$. 
Therefore, we see that the right $T$-action on $S\times T_{\C}$ is Hamiltonian with moment map $\tilde \mu = \mu \circ \beta.$
We have $f = \frac12 \langle y^{\frak t}, F y^{\frak t}\rangle = \frac12 \langle \mu, F \mu\rangle$. 

Let now $(u,z)$ be local coordinates on $S\times T_{\C}$, where $z = e^{a+i\theta}$ are standard holomorphic coordinates on $T_{\C}$, with $\theta$ the standard angular coordinate on $T$. Since the $T$-action on $S\times T_{\C}$ is the standard one,  
the Hamiltonian flow of $\tilde f = f \circ \beta = \frac12 \langle \tilde \mu, F \tilde \mu\rangle$ at time $s\in \R$ is then given by (see Proposition 
\ref{partialpullback})
$$
\bigl(u,e^{\sum_{j=1}^r(\theta_j+ia_j)T_j}\bigr) \mapsto \biggl(u, e^{\sum_{j=1}^r \bigl(\theta_j + s \frac{\partial \tilde f}{\partial \tilde \mu_j} +ia_j\bigr)T_j}\biggr).
$$ 
Its analytic continuation to imaginary time $\sigma = \sigma_1 + i\sigma_2, \sigma_i \in {\R}$, which preserves the $T_{\C}$-orbits, as in \cite{MN}, is then 
given by 
$$
\bigl(u,e^{\sum_{j=1}^r (\theta_j+ia_j)T_j}\bigr) \mapsto \bigl(u, e^{\sum_{j=1}^r (\theta_j + \sigma F\tilde \mu_j +ia_j)T_j}\bigr) = 
\bigl(u, e^{\sum_{j=1}^r (\theta_j +  \sigma_1 F\tilde \mu_j + i(a_j+\sigma_2 F\tilde\mu_j) )T_j}\bigr).
$$

We will now use some of the results of \cite{BG}. We will take $U\subset S$ to be a convex subset of some coordinate chart. From Section 2 in \cite{BG}, it follows that, on $U\times T_{\C}$,
$$
\tilde \omega = i \partial \bar\partial \rho,
$$  
where $\rho = \rho(u,a)$ is a strictly pluri-subharmonic function. Moreover, one then has the Legendre transform
$$
\tilde \mu = L_\rho(u,a) = \frac12 \frac{\partial \rho}{\partial a}.
$$
The inverse Legendre transform is given by a ``partial symplectic potential''
$$
g(u,\tilde \mu) = \tilde \mu a -\frac12\rho,
$$
with
$$
a = L_g(u,\tilde \mu)=\frac{\partial g}{\partial \tilde \mu}.
$$

This implies that the composition of Legendre transforms
$$
a \mapsto L_{g+\sigma_2 \tilde f}\circ L_g^{-1}(u,a)= a+\sigma_2 F \tilde \mu (u,a)
$$
is a diffeomorphism from $i\frak t \to i\frak t \cong {\R}^r$ (see also \cite{HS,F}). Indeed, explicitly,
$$
a+\sigma_2 F \tilde \mu (u,a) =  a'+\sigma_2 F \tilde \mu (u,a')
$$
implies
% $$
% a'-a = \sigma_2 F (\tilde \mu (u,a') -\tilde\mu(u,a)) = -\frac{\sigma_2}{2} F \left(\int_0^1 \frac{\partial^2 \rho}{\partial a^2}(a+t(a'-a))dt\right) \cdot (a'-a). 
% $$
\begin{IEEEeqnarray*}{rCls}
    a'-a & = & -\sigma_2 F (\tilde \mu (u,a') -\tilde\mu(u,a)) \\
    & = & -\frac{\sigma_2}{2} F \left(\int_0^1 \frac{\partial^2 \rho}{\partial a^2}(a+t(a'-a))dt\right) \cdot (a'-a). 
\end{IEEEeqnarray*}
By taking the inner product with $(a'-a)$ we see that, given that $\rho$ is strictly pluri-subharmonic, for $\sigma_2>0$ and $F$ 
positive definite this implies $a=a'$.
Hence, the analytic continuation of the Hamiltonian flow of $\tilde f$ is a diffeomorphism of $S\times T_{\C}$; since it preserves $T_{\C}$ orbits we conclude that 
it corresponds to a global diffeomorphism of $G_{\C}$. On the other hand, by looking at the formalism of \cite{MN} and at the action of $A_{\tau,\sigma}$ we see 
that this is equivalent to the statement that $A_{\tau,\sigma}$ is a global diffeomorphism for $\sigma_2>0$ and $\tau_2\neq 0$.
\end{proof}

\begin{remark}Note that the proof of Theorem \ref{sure} works if $f$ is replaced by any strictly convex function of $\mu$ so that the Theorem generalizes to that more general situation.
\end{remark}

\section{New polarizations of $T^*G$ from $G\times T$-invariant Hamiltonian flow}
\label{newpol}

\subsection{The polarizations ${\mathcal P}_{\tau,\sigma}$ and $G\times T$-invariant K\"ahler structures}

Recall that a polarization of $T^*G$, in the sense of geometric quantization, is an involutive Lagrangian distribution ${\mathcal P}$ in the complexified tangent bundle $T(T^*G) \otimes \C.$

Recall, also, from Section \ref{sectionquadratic} that we have two commuting Hamiltonian flows $\phi_{X_h}^t$ and $\phi_{X_f}^s$, for $t,s\in \R.$
The Hamiltonian flow $\phi_{X_h}^t$, analytically continued to complex time $\tau\in \C^+$, acting by push-forward 
on the vertical polarization ${\mathcal P}_{0,0}$, produces the $G\times G$-invariant K\"ahler structures of proposition \ref{cstau}. We will now act with both Hamiltonian flows in imaginary time to define new $G\times T$-invariant K\"ahler structures on $T^*G$. 

Recall that the vertical polarization ${\mathcal P}_{0,0}= {\rm Ker}\, D\pi$ is spanned by the Hamiltonian vector fields of Hamiltonian functions of the form 
$\pi^*f,\, f\in C^\infty(G)$, where $\pi\colon T^*G \to G$ is the canonical projection. To study the push-forward of ${\mathcal P}_{0,0}$ under Hamiltonian flow it is then useful to recall the following. Under the Hamiltonian flow of $X_g, g\in C^\infty(T^*G)$, one has
that the evolution of the Hamiltonian vector field of $f\in C^\infty(T^*G)$ is given by
$$
X_{(\phi^t_{X_g})^* f} = (\phi^{-t}_{X_g})_*\, X_f. 
$$
When the Hamiltonian flow $\phi^t_{X_g}$ is real analytic in $t$ one has, moreover, within the convergence regions for the power series, the corresponding formulas
$$
X_{e^{t X_g}f} = e^{t {\mathcal L}_{X_g}} X_f,
$$  
where ${\mathcal L}_{X_g}$ denotes the Lie derivative along $X_g.$ (See, for example, \cite{HK, KMN1, MN, P}.)

\begin{proposition}\label{polsts}
Let ${\mathcal P}_{t,s}$ be the real polarization of $T^*G$ obtained by push-forward of ${\mathcal P}_{0,0}$ by the combined Hamiltonian flow $\phi_{X_f}^{-s} \circ \phi_{X_h}^{-t}$, for $t,s \in \R$. 
$$
{\mathcal P}_{t,s} = \left( \phi_{X_f}^{-s} \circ \phi_{X_h}^{-t}\right)_* {\mathcal P}_{0,0}, \quad t,s\in \R.
$$
This polarization can described by sections analytic in $t$ and $s$ as follows
\begin{equation*}
{\mathcal P}_{t,s} = \spn_{\C} \biggl\{ e^{t {\mathcal L}_{X_h}}\circ e^{s {\mathcal L}_{X_f}} \frac{\partial}{\partial y^j} \biggm| j=1, \dots, n \biggr\}.
\end{equation*}
Explicitly, we have
\begin{IEEEeqnarray}{rCl}
		\nonumber \mathcal{P}_{t, s}|_{(x,y)} &=& \left\{ \left( \bigg[\frac{1 - e^{t \ad_{u_h(y)}}}{\ad_{u_h(y)}} H_h(y) e^{s \, \ad_{Fy}} - s F \bigg]A, \bigg[ e^{s \, \ad_{Fy}} - s \, \ad_y \circ F \bigg] A \right)  \Bigg|  A \in \g_{\C} \right\}  \\
		\label{eq:ptausigmaz} &=& \left\{ \left( \bigg[\frac{1 - e^{t \ad_{u_h(y)}}}{\ad_{u_h(y)}} H_h(y) -s  e^{t \ad_{u_h(y)}} F \bigg]A, A \right)  \Bigg|  A \in \g_{\C}  \right\}. 
	\end{IEEEeqnarray}
\end{proposition}

\begin{proof}
	Equation (\ref{eq:ptausigmaz}) follows immediately from the definition of $\mathcal{P}_{t, s}$, Lemma \ref{composederivatives}, Lemma \ref{invertible} and from noting that
	\begin{equation*}
		\left( \frac{1 - e^{t \ad_{u_h(y)}}}{\ad_{u_h(y)}} H_h(y) -s e^{t \ad_{u_h(y)}} F \right) \left( e^{s \, \ad_{Fy}} -s  \ad_y \circ F \right) = \frac{1 - 
		e^{t \ad_{u_h(y)}}}{\ad_{u_h(y)}} H_h(y) e^{s \, \ad_{Fy}} -s  F, 
	\end{equation*}
	where one uses (\ref{commutes}) and the fact that 
	$$
	F \circ e^{s \, \ad_{Fy}} = F,\,\,\, F \circ \ad_y \circ F =0.
	$$
	The claim that ${\mathcal P}_{t,s}$ can be obtained by exponentiating the Lie derivatives along $X_h$ and $X_f$ follows from Lemma 
	\ref{lema2apen} in the Appendix. 
\end{proof}

Proposition \ref{polsts} has the immediate

\begin{corollary}The two real-parameter family of polarizations $\{{\mathcal P}_{t,s}\}_{t,s\in {\R}}$ 
extends to a two-complex parameter family of polarizations\footnote{Recall that, from \cite{MN} (see also \cite{KMN} for the case $\sigma=0$) holomorphic vector fields for the complex structures generated in this process are obtained by analytic continuation $t\to \bar\tau, s\to \bar\sigma$.} $\{{\mathcal P}_{\tau,\sigma}\}_{\tau,\sigma\in {\C}}$ by (unique) analytic continuation, 
in $t$ to $\bar\tau = \tau_1-i\tau_2 \in {\C}$ and in $s$ to $\bar\sigma = \sigma_1-i\sigma_2\in {\C}$, with $\tau_1, \tau_2, \sigma_1, \sigma_2\in {\R}$, pointwise along $T(T^*G)\otimes{\C}$. Explicitly,
\begin{IEEEeqnarray}{rCl}
		 \nonumber\mathcal{P}_{\tau, \sigma}|_{(x,y)} & = & \left\{ \left( \bigg[\frac{1 - e^{\bar\tau \ad_{u_h(y)}}}{\ad_{u_h(y)}} H_h(y) e^{\bar \sigma \ad_{Fy}} - \bar\sigma F \bigg]A, \bigg[ e^{\bar\sigma \ad_{Fy}} - \bar\sigma  \ad_y \circ F \bigg] A \right)  \Bigg| A \in \g_{\C} \right\} \\
		\nonumber & = & \left\{ \left( \bigg[\frac{1 - e^{\bar\tau \ad_{u_h(y)}}}{\ad_{u_h(y)}} H_h(y) -\bar\sigma  e^{\bar\tau \ad_{u_h(y)}} F \bigg]A, A \right) \Bigg| A \in \g_{\C}  \right\}. 
	\end{IEEEeqnarray}
\end{corollary}

We will now establish that this family of $G\times T$-invariant complex polarizations of $T^*G$, obtained by analytic continuation of the above Hamiltonian flows to imaginary time, contains, in fact, a family of K\"ahler polarizations. 

\begin{theorem}\label{super}The polarizations ${\mathcal P}_{\tau,\sigma}$, for $\tau\in {\C}^+, \sigma \in {\C}^+\cup {\R}$, are obtained by pull-back via $A_{\tau,s}$, $s\in \R$, of the holomorphic tangent space of $G_{\C}$ with respect to the standard complex structure, $T^{(1,0)}G_{\C}$, followed by analytic continuation in $s$ to $\bar\sigma$. 
\end{theorem}

\begin{proof}Recall that, for $s\in {\R}$, $A_{\tau,s}=\psi_\tau \circ \alpha_h \circ \phi^s_{X_f}$. Also, from \cite{KMN}, by setting $\sigma=0$ in (\ref{eq:ptausigmaz}), we obtain, for $\tau\in {\C}^+$,
$$
{\mathcal P}_{\tau,0} = (\psi_\tau \circ \alpha_h)^* T^{(1,0)}G_{\C}.
$$ 
Therefore, to prove the theorem it is enough to show that, for $s\in {\R}$,
\begin{equation}\label{enough}
{\mathcal P}_{\tau,s}\vert_{(x,y)} = D\phi^{-s}_{X_f} \Big( {\mathcal P}_{\tau,0}\vert_{\phi^s_{X_f}(x,y)} \Big),
\end{equation}
and then to take the (unique) analytic continuation in $s$ to $\bar \sigma$. To prove (\ref{enough}), just use (\ref{eq:ptausigmaz}) with $s=0$, Lemma 
\ref{composederivatives}, equality (\ref{adhessian}) and the well-known identities
\begin{equation*}
    \ad_{\Ad_gA} = \Ad_g \circ \ad_A \circ \Ad_{g^{-1}},\quad \Ad_{\Ad_g \tilde g}= \Ad_g \circ \Ad_{\tilde g} \circ \Ad_{g^{-1}},\quad g,\tilde g\in G, A\in \g. \qedhere
\end{equation*}
% $$
%\ad_{\Ad_gA} = \Ad_g \circ \ad_A \circ \Ad_{g^{-1}},\, \Ad_{\Ad_g \tilde g}= \Ad_g \circ \Ad_{\tilde g} \circ \Ad_{g^{-1}},\,\, g,\tilde g\in G, A\in \g.
%$$ 
\end{proof}

Let $J_{\tau,\sigma}$ be the ($G\times T$-invariant) complex structure on $T^*G$ defined by the pull-back by $A_{\tau,\sigma}$ of the standard complex structure on $G_{\C}$, so that 
$$
{\mathcal P}_{\tau,\sigma} = T^{(1,0)}(T^*G, J_{\tau,\sigma}).
$$

\begin{theorem}\label{ohyes}For $\tau\in {\C}^+, \sigma\in {\C}^+\cup {\R}$, $(T^*G,\omega,J_{\tau,\sigma})$ is K\"ahler. A (global) K\"ahler potential is given by
\begin{equation}\label{kahlerpot}
\kappa_{\tau,\sigma} (x,y) = 2\tau_2 \left(\langle y,u_h(y)\rangle -h(y)\right) +2 \sigma_2 f.
\end{equation}
\end{theorem}

\begin{proof}
The fact that ${\mathcal P}_{\tau,\sigma}$ is compatible with $\omega$, that is,
$$
\omega_{\vert_{{\mathcal P}_{\tau,\sigma}\times{\mathcal P}_{\tau,\sigma}}}= \omega_{\vert_{\bar{\mathcal P}_{\tau,\sigma}\times\bar{\mathcal P}_{\tau,\sigma}}} =0,
$$
follows directly from the proof of Theorem 4.1 in \cite{MN} (this is a purely local argument where the fact that $T^*G$ is not compact is irrelevant) and can also be easily checked by direct calculation. Positivity is equivalent to
$$
i\omega_{\vert_{\bar{\mathcal P}_{\tau,\sigma}\times{\mathcal P}_{\tau,\sigma}}} >0.
$$
Let 
\begin{IEEEeqnarray*}{rCls}
    M^{\tau,\sigma} &=& \biggl[\frac{1 - e^{\bar\tau \ad_{u_h(y)}}}{\ad_{u_h(y)}} H|_y e^{\bar \sigma \ad_{Fy}} - \bar\sigma F \biggr],\\
    N^{\tau,\sigma}&=& \bigl[ e^{\bar\sigma \ad_{Fy}} - \bar\sigma  \ad_y \circ F \bigr],
\end{IEEEeqnarray*}
\begin{equation}\label{frame}
E_j^{\tau,\sigma} = \left(M^{\tau,\sigma}(T_j), N^{\tau,\sigma} (T_j) \right),\quad j=1,\dots, n.
\end{equation}
We have, from (\ref{eq:omega2}),
% \begin{IEEEeqnarray}{rCl}
%     i \omega|_{(x,y)} (\overline{E_j^{\tau,\sigma}},E_k^{\tau,\sigma}) & = & i \Big[ \langle \overline{M^{\tau,\sigma}(y)}T_j , N^{\tau,\sigma}(y) T_k \rangle \nonumber\\
% 		& & \quad {}- \langle \overline{N^{\tau,\sigma}(y)} T_j, M^{\tau,\sigma}(y) T_k \rangle \nonumber \\
% 		& & \quad {}+ \langle \ad_y \overline{M^{\tau,\sigma}(y)} T_j , M^{\tau,\sigma}(y) T_k \rangle \Big] \\
% 		& = & i \Big[ N^{\tau,\sigma}(y)^\top \overline{M^{\tau,\sigma}(y)} - M^{\tau,\sigma}(y)^\top \overline{N^{\tau,\sigma}(y)} \nonumber \\
% 		& & \quad {}+ M^{\tau,\sigma}(y)^\top \ad_y \overline{M^{\tau,\sigma}(y)} \Big]^j_k
% \end{IEEEeqnarray}
\begin{IEEEeqnarray*}{rCl}
    \IEEEeqnarraymulticol{3}{l}{i \omega|_{(x,y)} (\overline{E_j^{\tau,\sigma}},E_k^{\tau,\sigma})} \\ 
    \quad & = & i \Big[ \langle \overline{M^{\tau,\sigma}(y)}T_j , N^{\tau,\sigma}(y) T_k \rangle- \langle \overline{N^{\tau,\sigma}(y)} T_j, M^{\tau,\sigma}(y) T_k \rangle \\ 
    \quad &  & \hphantom{\Big[} {}+ \langle \ad_y \overline{M^{\tau,\sigma}(y)} T_j , M^{\tau,\sigma}(y) T_k \rangle \Big] \\ 
	\quad & = & i \Big[ N^{\tau,\sigma}(y)^\top \overline{M^{\tau,\sigma}(y)} - M^{\tau,\sigma}(y)^\top \overline{N^{\tau,\sigma}(y)} + M^{\tau,\sigma}(y)^\top \ad_y \overline{M^{\tau,\sigma}(y)} \Big]^j_k
	\end{IEEEeqnarray*}
	Define $W^{\tau,\sigma} = i N^{\tau,\sigma}(y)^\top \overline{M^{\tau,\sigma}(y)} - i M^{\tau,\sigma}(y)^\top \overline{N^{\tau,\sigma}(y)} + i M^{\tau,\sigma}(y)^\top \ad_y \overline{M^{\tau,\sigma}(y)}$. Using the properties of $F$ and $h$, namely
	$$
	[\ad_y,H_h(y)] = [\ad_{u_h(y)}, H_h(y)] =0, \quad F \circ e^{\sigma \ad_{Fy}} = e^{\sigma \ad_{Fy}} \circ F = F,\quad F\circ \ad_y \circ F =0,
	$$
	we obtain
	\begin{equation}\nonumber
		W^{\tau,\sigma} = e^{-\bar\sigma \ad_{Fy}} \left( i \frac{1 - e^{2 i\tau_2 \ad_{u_h(y)}}}{\ad_{u_h(y)}} H_h(y) \right) e^{{\sigma} \ad_{Fy}} + 2 \sigma_2 F.
	\end{equation}
	From equation (3.7) in \cite{KMN} we obtain that the matrix
	$$
	C=\left( i \frac{1 - e^{2 i\tau_2 \ad_{u_h(y)}}}{\ad_{u_h(y)}} H_h(y) \right)
	$$
is positive definite for $\tau_2>0$, so that 
$$
e^{-\bar\sigma \ad_{Fy}} C e^{{\sigma} \ad_{Fy}} = \left(e^{{\sigma} \ad_{Fy}}\right)^\dagger C e^{{\sigma} \ad_{Fy}}
$$
is also positive definite. Therefore for $\tau\in {\C}^+, \sigma_2 \geq 0$ we obtain that ${\mathcal P}_{\tau,\sigma}$ is a K\"ahler polarization. 

Equation (\ref{kahlerpot}) for the K\"ahler potential can be obtained by Theorem 4.1 in \cite{MN} (where, again, only a local argument is used and the fact that 
$T^*G$ is non-compact is irrelevant). Alternatively, one can check explicitly, using (\ref{theta}) and the invariance of the inner product on $\frak g$, that
\begin{IEEEeqnarray*}{rCls}
    \theta (E_j^{\tau,\sigma}) &=& -\bar\tau \langle y, H_h(y) e^{\bar\sigma \ad_{Fy}} T_j\rangle - \langle y, \bar\sigma FT_j\rangle\\
    &=& d\bar\lambda_{\tau,\sigma} (E_j^{\tau,\sigma}),\\
\end{IEEEeqnarray*}
%$$
%\theta (E_j^{\tau,\sigma}) = -\bar\tau \langle y, H_h(y) e^{\bar\sigma \ad_{Fy}} T_j\rangle - \langle y, \bar\sigma FT_j\rangle = d\bar\lambda_{\tau,\sigma} (E_j^{\tau,\sigma}),
%$$
where
$\bar \lambda_{\tau,\sigma} = -\bar\tau (\langle y,u_h(y)\rangle - h(y)) -\bar\sigma f$, so that $\kappa_{\tau,\sigma} = 2 {\rm Im} \bar\lambda_{\tau,\sigma}$ 
is a K\"ahler potential, as claimed.
\end{proof}

Let now $\tau\in {\C}^+, \sigma\in {\C}^+\cup {\R}$, and consider the left $G$-invariant holomorphic trivializing frame for the canonical bundle of $(T^*G,J_{\tau,\sigma})$ given by
\begin{equation}\label{sectioncanonical}
\Omega_{\tau,\sigma} = e^{\sigma {\mathcal L}_{X_f}} \Omega_{\tau,0},
\end{equation}
where $\Omega_{\tau,0}$ is the left $G$-invariant trivializing section for the 
canonical bundle of $(T^*G,J_{\tau,0})$ which is described in Theorem 3.10 of \cite{KMN1}, namely
$$
\Omega_{\tau,0} = e^{\tau {\mathcal L}_{X_h}} w^1\wedge \cdots \wedge w^n. 
$$
Let $\sqrt{\Omega_{\tau,\sigma}}$ be a trivializing section of the bundle of half-forms, $\sqrt{K_{{\mathcal P}_{\tau,\sigma}}}$. 
(Here, a preferred choice of square root of the canonical bundle has been made, namely we take $\sqrt{K_{{\mathcal P}_{\tau,\sigma}}}$ to be trivializable where the trivializing section $\sqrt{\Omega_{\tau,\sigma}}$ is left $G$-invariant. See the Appendix in \cite{KMN1} for a more detailed discussion.)
The half-form correction will then be given by  
\begin{equation}\label{normhalf}
\vert\sqrt{\Omega_{\tau,\sigma}} \vert^2 = \sqrt{\frac{\bar\Omega_{\tau,\sigma}\wedge \Omega_{\tau,\sigma}}{(2i)^n (-1)^{n(n-1)/2} \omega^n/n!}}.
\end{equation}

>From Lemma 4.3 in \cite{KMN} we recall
$$
\vert\sqrt{\Omega_{\tau,0}} \vert^2 = \tau_2^{\frac{n}{2}} \eta(\tau_2u_h(y)) (\det H_h(y))^\frac12,
$$
where $\eta$ is the $\Ad$-invariant function on $\frak g$ which is defined for $y\in \frak t$ by
$$
\eta(y) = \Pi_{\alpha\in \Delta^+} \frac{\sinh {\alpha(y)}}{\alpha(y)},
$$
where $\Delta^+$ is the corresponding set of positive roots.
  
\begin{proposition}One has, for $\tau\in {\C}^+, \sigma \in {\C}^+\cup{\R}$,
 \begin{equation*}
 \vert\sqrt{\Omega_{\tau,\sigma}} \vert^2 =  \vert\sqrt{\Omega_{\tau,0}} \vert^2 \det\left(1+\sigma_2 \frac{\ad_{u_h(y)}e^{i\tau_2 \ad_{u_h(y)}}}{\sin \tau_2 
 \ad_{u_h(y)}}  F H_h(y)^{-1} \right)^{-\frac12}.
\end{equation*}
\end{proposition}

\begin{proof} 
This follows by direct evaluation. Let $DA_{\tau,\sigma}, \sigma\in {\C}$, denote the (unique) analytic continuation of $DA_{\tau,s}, s\in {\R}$, from $s$ 
to $\sigma$. An holomorphic frame $\{Z^j_{\tau,\sigma}\}_{j=1,\dots , n}$ for ${\mathcal P}_{\tau,\sigma}$ can be obtained by applying 
$$
DA_{\tau,\bar\sigma}^{-1} 
$$
to an holomorphic frame on $G_{\C}$ given by the columns of the $2n\times n$ matrix
$$
\frac12 \left[
\begin{array}{c}
I \\
-i I
\end{array} \right].
$$
For $\sigma =0$ this coincides with the frame given in \cite{KMN13}.
One then obtains
$$
Z_{\tau,\sigma}^1 \wedge \cdots \wedge Z_{\tau,\sigma}^n \wedge \bar Z_{\tau,\sigma}^1 \wedge \cdots \wedge \bar Z_{\tau,\sigma}^n = 
(\det O) X_1 \wedge \cdots X_n \wedge \frac{\partial}{\partial y^1} \wedge \cdots \wedge \frac{\partial}{\partial y^n},
$$
where $O$ is the $2n\times 2n$ matrix
$$
O=\left[
\begin{array}{cc}
A & B\\
C & D
\end{array}\right] \cdot
\left[\begin{array}{cc}
\Ad_{e^{\bar\sigma Fy}} & 0\\
0 & \Ad_{e^{\sigma Fy}}
\end{array}\right],
$$
where the $n\times n$ blocks are given by
% $$
% A =\scriptsize{
% \left(1-i\frac{e^{\tau_1 \ad_{u_h(y)}}}{\sin \tau_2 \ad_{u_h(y)}} (e^{-\tau_1 \ad_{u_h(y)}}\cos(\tau_2 \ad_{u_h(y)})-1)\right) -\bar\sigma F \ad_y \left(1-i\frac{\cos\tau_2 \ad_{u_h(y)}}{\sin \tau_2 \ad_{u_h(y)}}\right)},
% $$
% $$
% B =\scriptsize{\left(1+i\frac{e^{\tau_1 \ad_{u_h(y)}}}{\sin \tau_2 \ad_{u_h(y)}} (e^{-\tau_1 \ad_{u_h(y)}}\cos(\tau_2 \ad_{u_h(y)})-1)\right) -\sigma F \ad_y \left(1+i\frac{\cos\tau_2 \ad_{u_h(y)}}{\sin \tau_2 \ad_{u_h(y)}}\right)},
% $$
% $$
% C = \scriptsize{(1-\bar\sigma \ad_y \circ F) \ad_y \left(1-i\frac{\cos\tau_2 \ad_{u_h(y)}}{\sin \tau_2 \ad_{u_h(y)}}\right)}
% $$
% and
% $$
% D = \scriptsize{(1-\sigma \ad_y \circ F) \ad_y \left(1+i\frac{\cos\tau_2 \ad_{u_h(y)}}{\sin \tau_2 \ad_{u_h(y)}}\right)}.
% $$
\begin{IEEEeqnarray*}{rCls}
    A & = & \left(1-i\frac{e^{\tau_1 \ad_{u_h(y)}}}{\sin \tau_2 \ad_{u_h(y)}} (e^{-\tau_1 \ad_{u_h(y)}}\cos(\tau_2 \ad_{u_h(y)})-1)\right) -\bar\sigma F \ad_y \left(1-i\frac{\cos\tau_2 \ad_{u_h(y)}}{\sin \tau_2 \ad_{u_h(y)}}\right), \\
B & = & \left(1+i\frac{e^{\tau_1 \ad_{u_h(y)}}}{\sin \tau_2 \ad_{u_h(y)}} (e^{-\tau_1 \ad_{u_h(y)}}\cos(\tau_2 \ad_{u_h(y)})-1)\right) -\sigma F \ad_y \left(1+i\frac{\cos\tau_2 \ad_{u_h(y)}}{\sin \tau_2 \ad_{u_h(y)}}\right),\\
C & = & (1-\bar\sigma \ad_y \circ F) \ad_y \left(1-i\frac{\cos\tau_2 \ad_{u_h(y)}}{\sin \tau_2 \ad_{u_h(y)}}\right),\\
D & = & (1-\sigma \ad_y \circ F) \ad_y \left(1+i\frac{\cos\tau_2 \ad_{u_h(y)}}{\sin \tau_2 \ad_{u_h(y)}}\right).
\end{IEEEeqnarray*}
The result then follows by using the properties of $F$ and $\ad_y$ to evaluate the determinant and by (\ref{normhalf}). 
\end{proof}

We then obtain the following immediate

\begin{corollary}\label{half-formgrowth}There are positive constants $c_0,c_1$ such that
$$
\vert\sqrt{\Omega_{\tau,\sigma}} \vert^2 \leq c_0 e^{{c_1} \vert\vert y \vert\vert}.
$$
\end{corollary}

\subsection{The mixed polarizations ${\mathcal P}_{0,\sigma}$ and partial K\"ahler structures}

By setting $\tau =0$ in the setting of the previous Section we obtain a family of mixed $G\times T$-invariant polarizations ${\mathcal P}_{0,\sigma}$. The superscripts $\frak t$ and ${\frak t}^\perp$ will denote components of vectors along $\frak t$ and ${\frak t}^\perp$, respectively.

Let then,
\begin{IEEEeqnarray}{rCl}
		\nonumber \mathcal{P}_{0, \sigma}|_{(x,y)} & = & \Bigl\{ \left(  - \bar\sigma F A, \bigl[ e^{\bar\sigma \ad_{Fy}} - \bar\sigma  \ad_y \circ F \bigr] A \right) \Bigm| A \in \g_{\C} \Bigr\} \IEEEeqnarraynumspace \\
		\nonumber & = & \left\{ \left( -\bar\sigma   F A, A \right) \ |\  A \in \g_{\C}  \right\}. \IEEEeqnarraynumspace
	\end{IEEEeqnarray}

We have
$$
\mathcal{P}_{0, \sigma} = \mathcal{P}_{0, \sigma}^{\frak t}  \oplus \mathcal{P}_{0, \sigma}^{{\frak t}^\perp},
$$
where
% $$
% \mathcal{P}_{0, \sigma}^{\frak t} = \left\{ \left( -\bar\sigma   FA, A \right) \bigm|  A \in {\frak t}_{\C}  \right\}
% $$
% and
% $$
% \mathcal{P}_{0, \sigma}^{{\frak t}^\perp} = \left\{ \left( 0, A \right) \bigm|  A \in {\frak t}^\perp_{\C}  \right\}.
% $$
\begin{IEEEeqnarray*}{rCls}
    \mathcal{P}_{0, \sigma}^{\frak t} & = &\left\{ \left( -\bar\sigma   FA, A \right) \bigm|  A \in {\frak t}_{\C}  \right\},\\
    \mathcal{P}_{0, \sigma}^{{\frak t}^\perp} &= &\left\{ \left( 0, A \right) \bigm|  A \in {\frak t}^\perp_{\C}  \right\}.
\end{IEEEeqnarray*}
Note that for $\sigma\neq 0$, we also have
$$
\mathcal{P}_{0, \sigma}^{\frak t} = \left\{ \left( A, -\bar\sigma^{-1} F^{-1}A \right) \bigm| A \in {\frak t}_{\C}  \right\}.
$$

\begin{lemma}
One has, for $\sigma_2>0$,
$$
\overline{\mathcal{P}_{0, \sigma}^{\frak t}} \cap \mathcal{P}_{0, \sigma}^{\frak t} =\left\{0\right\},\quad \overline{\mathcal{P}_{0, \sigma}^{{\frak t}^\perp}} = 
\mathcal{P}_{0, \sigma}^{{\frak t}^\perp}.
$$
\end{lemma}

\begin{proof}This follows immediately from the fact that $f$ is a linear isomorphism and from $\sigma_2 >0$.
\end{proof}

We see that the polarization ${\mathcal P}_{0,\sigma}$ is mixed. Remarkably, we will now see that it is associated with a foliation of $T^*G$ by submanifolds with
K\"ahler structure defined by $\mathcal{P}_{0, \sigma}^{\frak t}$. 

Let $p\colon G\to G/T$ be the principal fiber bundle over the flag manifold $G/T$ obtained by the right action of $T$ on $G$. Let $F_x$ be the fiber of $p$ containing $x\in G$. Note that $F_x \hookrightarrow G$ is an embedded submanifold which is (non-canonically) diffeomorphic to $T$. 

\begin{theorem}For $\sigma_2>0$, the distribution ${\mathcal P}_{0,\sigma}^{\frak t}$ defines a K\"ahler structure along the symplectic 
submanifolds $L_{(x,y)}=F_x\times \left\{y+\frak t\right\} \stackrel{\iota}{\hookrightarrow} T^*G,\,\, (x,y)\in T^*G$. 
A global K\"ahler potential along $L_{(x,y)}$ is given by
$$
\kappa_{0,\sigma} = 2\sigma_2 (\iota^*f).
$$
\end{theorem}

\begin{proof}
Let $\Sigma_\sigma\subset T(T^*G)$ be the (real) distribution defined by
$$
\Sigma_{\sigma} \otimes {\C} = ({\rm Re}{\mathcal P}_{0,\sigma}^{\frak t} \oplus {\rm Im}{\mathcal P}_{0,\sigma}^{\frak t})\otimes {\C} = 
{\mathcal P}_{0,\sigma}^{\frak t}\oplus \overline{{\mathcal P}_{0,\sigma}^{\frak t}}.
$$
It is straighforward to check that $\Sigma_\sigma$ is involutive whence it defines a foliation of $T^*G$. It is also easy to verify that 
$\Sigma_\sigma = \frak t \oplus \frak t$ so that indeed 
$$
T_{(x,y)}L_{(x,y)} = \Sigma_\sigma\vert_{(x,y)}.
$$
Now, one can take the frame (\ref{frame}) for $\tau =0$ and ordering the basis $\left\{T_j\right\}_{j=1,\dots, n}$ so that $\left\{T_j\right\}_{j=1,\dots, r}$
is a basis of $\frak t$, to get 
$$
{\mathcal P}_{0,\sigma}^{\frak t} = \spn_{\C} \Bigl\{ {E_j^{0,\sigma}}^{\frak t} \Bigm| j=1,\dots, r\Bigr\}.
$$
As in Theorem \ref{ohyes}, one can checks that, for $\sigma_2>0$,
$$
\omega\vert_{{{\mathcal P}_{0,\sigma}^{\frak t}}\times {\mathcal P}_{0,\sigma}^{\frak t}}=0,
$$
and
$$
i\omega\vert_{\overline{{\mathcal P}_{0,\sigma}^{\frak t}}\times {\mathcal P}_{0,\sigma}^{\frak t}} >0,
$$ 
so that, indeed, the leaves $L_{(x,y)}$ are K\"ahler. Note that $\iota^*\theta$ is a potential for $\iota^*\omega$ and that
$$
\iota^*\theta \Bigl({E^{0,\sigma}_j}^{\frak t}\Bigr) = -\bar\sigma df \Bigl({E^{0,\sigma}_j}^{\frak t}\Bigr), \quad j=1, \dots, r,
$$
so that, as in the proof of Theorem \ref{ohyes}, $2\sigma_2 f\circ \iota$ is a K\"ahler potential for $\iota^*\omega$.
\end{proof}

We see that the complexifier $f$, being convex only ``along the directions of the maximal torus $T$'', generates, by push-forward of the vertical polarization, a 
``$G\times T$-invariant'' foliation of $T^*G$ by K\"ahler manifolds, each of these being diffeomorphic to $T_{\C}\cong T^*T \cong T\times \frak t$.

\begin{remark}In the notation of \cite{Wo} (see Chapter 5) we have
% $$
% D_\sigma = {\mathcal P}_{0,\sigma} \cap \overline{{\mathcal P}_{0,\sigma}} \cap T(T^*G) =  {\mathcal P}_{0,\sigma}^{{\frak t}^\perp} = \langle
% \frac{\partial}{\partial y_{r+1}},\dots, \frac{\partial}{\partial y_{n}}\rangle_{\R},
% $$
\begin{IEEEeqnarray*}{rCls}
    D_\sigma & = & {\mathcal P}_{0,\sigma} \cap \overline{{\mathcal P}_{0,\sigma}} \cap T(T^*G) \\
    & = & {\mathcal P}_{0,\sigma}^{{\frak t}^\perp} \\
    & = & \spn_{\R} \biggl\{ \frac{\partial}{\partial y_{r+1}},\dots, \frac{\partial}{\partial y_{n}}\biggr\},
\end{IEEEeqnarray*}
and
\begin{IEEEeqnarray*}{rCls}
    E_\sigma & = & ({\mathcal P}_{0,\sigma} \oplus \overline{{\mathcal P}_{0,\sigma}}) \cap T(T^*G) \\
    & = & \Sigma_\sigma \oplus D_\sigma.
\end{IEEEeqnarray*}
% $$
% E_\sigma = ({\mathcal P}_{0,\sigma} \oplus \overline{{\mathcal P}_{0,\sigma}}) \cap T(T^*G) = \Sigma_\sigma \oplus D_\sigma.
% $$
Note that $E_\sigma$ is involutive so that ${\mathcal P}_{0,\sigma}$ is strongly integrable in the sense of \cite{Wo}. Let $\tilde L_{(x,y)}$ be the leaf of 
$E_\sigma$ through $(x,y)\in T^*G$. Then,
$$
\tilde L_{(x,y)} = F_x \times \frak g.
$$
The (K\"ahler) leaves of $\Sigma_\sigma$ are then given by 
$$
L_{(x,y)} = \tilde L_{(x,y)} /D_\sigma,
$$ 
so that $L_{(x,y)}$ is the, so-called, coisotropic reduction of the  coisotropic submanifold $\tilde L_{(x,y)}\subset T^*G$. (See Section 5 of \cite{Wo}). Note that, in this case, $\Sigma_\sigma$ is also involutive.
\end{remark}

Let $J_{0,\sigma}^{L_{(x,y)}}$ be the complex structure on the leaf $L_{(x,y)}$, so that
$$
{\mathcal P}_{0,\sigma}^{\frak t} \vert_{L_{(x,y)}}= T^{(1,0)} \Bigl(L_{(x,y)},J_{0,\sigma}^{L_{(x,y)}}\Bigr).
$$

For $x_0\in F_x$ and $\sigma_2>0$, consider the diffeomorphism
\begin{eqnarray}\label{beta}
\beta_{x_0}^\sigma\colon L_{(x,y)}=F_x\times \{y+\frak t\}&\to& T_{\C}\\ \nonumber
(x,y+a)&\mapsto& te^{\sigma Fa},
\end{eqnarray}
where $x= x_o\cdot t$, $t\in T$, $a\in \frak t$.
Note that if $x_0^{'}=x_0\cdot t^{'}\in F_x$ we have that $\beta_{x_0}^\sigma$ and $\beta_{x_0^{'}}^\sigma$ are related by a translation by $t^{'}\in T$ in $T_{\C}$.

We then have, as an analog of Theorem \ref{super},

\begin{proposition}\label{partialpullback}The complex structure $J_{0,\sigma}^{L_{(x,y)}}$ on $L_{(x,y)}$ is the pull-back of the standard complex 
structure $J_{st}^{T_{\C}}$ on $T_{\C}$ by $\beta_{x_0}^\sigma$. (Note that $T$-invariance of $J_{st}^{T_{\C}}$ ensures that the choice of $x_0\in F_x$ is irrelevant.) This defines an holomorphic action of $T_{\C}$ on $(L_{(x,y)},J_{0,\sigma}^{L_{(x,y)}})$.
\end{proposition}

\begin{proof}Let us take holomorphic coordinates for the standard complex structure $J_{st}^{T_{\C}}$ on $T_{\C}$,
$z_j = \theta_j + i y_j,$ with
$$
e^{\sum_{j=1}^r \theta_jT_j + i y_jT_j}\in T_{\C},
$$
and
$$
\frac{\partial}{\partial z_j} = \frac12 \left(\frac{\partial}{\partial \theta_j}- i \frac{\partial}{\partial y_j}\right), \quad j=1, \dots, r.
$$
In the basis $\Bigl\{\frac{\partial}{\partial \theta_j}, \frac{\partial}{\partial y_j}\Bigr\}_{j=1, \dots, r}$ we have
$$
D\beta_{x_0}^\sigma = \left[
\begin{array}{cc}
\id & \sigma_1 F \\
0 & \sigma_2 F
\end{array}\right],
$$
so that 
$$
D\beta_{x_0}^\sigma {E^{0,\sigma}_j}^{\frak t} = i\sigma_2(FT_j, -i FT_j ), \quad j=1, \dots, r.
$$
These are eigenvectors of $J_{st}^{T_{\C}}$ with eigenvalue $+i$ which proves the Proposition.
\end{proof}

Let then $z_{\sigma}^j, j=1,\dots, r,$ be holomorphic coordinates along the fibers $L_{x,y}$ obtained by pull-back via $\beta_{x_0}^\sigma$ of the standard holomorphic coordinates on $T_{\C}$, so that 
$$
{\mathcal P}_{0,\sigma}^{\frak t} = \spn_{\C} \biggl\{ \frac{\partial}{\partial z_\sigma^1}, \dots, \frac{\partial}{\partial z_\sigma^r} \biggr\}.
$$ 
>From (\ref{beta}) we have $dz_\sigma^j = w^j+\sigma \sum_{k=1}^r F_{jk} dy^k.$ Note that, for the symplectic forms along the K\"ahler leaves $L_{(x,y)}$ one has, denoting simply by $\partial$ the $\partial$-operator relative to $J_{0,\sigma}^{L_{(x,y)}}$,
% $$
% \iota^*\omega = \sum_{k=1}^r w^k\wedge dy^k = i\partial\bar\partial 2\sigma_2\iota^*f = \sum_{j,k=1}^r \frac{i}{2\sigma_2} F^{-1}_{jk}dz_\sigma^j\wedge 
% d\bar z_\sigma^k.
% $$
\begin{IEEEeqnarray*}{rCls}
    \iota^*\omega & = & \sum_{k=1}^r w^k\wedge dy^k \\
    & = & i\partial\bar\partial 2\sigma_2\iota^*f \\
    & = & \sum_{j,k=1}^r \frac{i}{2\sigma_2} F^{-1}_{jk}dz_\sigma^j\wedge d\bar z_\sigma^k.
\end{IEEEeqnarray*}
In particular,
% $$
% (\iota^*\omega)^r/r! = (-1)^{\frac{r(r-1)}{2}} w^1\wedge \cdots \wedge w^r \wedge dy^1 \wedge \cdots \wedge dy^r = (-1)^{\frac{r(r-1)}{2}} \left(\frac{i}{2\sigma_2}\right)^r
% \det{F^{-1}} dz_\sigma^1 \wedge \cdots d z_\sigma^r \wedge d\bar z_\sigma^1 \wedge \cdots \wedge d\bar z_\sigma^r.
% $$
\begin{IEEEeqnarray*}{rCls}
    (\iota^*\omega)^r/r! & = & (-1)^{\frac{r(r-1)}{2}} w^1\wedge \cdots \wedge w^r \wedge dy^1 \wedge \cdots \wedge dy^r \\
    & = & (-1)^{\frac{r(r-1)}{2}} \left(\frac{i}{2\sigma_2}\right)^r
\det{F^{-1}} dz_\sigma^1 \wedge \cdots \wedge d z_\sigma^r \wedge d\bar z_\sigma^1 \wedge \cdots \wedge d\bar z_\sigma^r.
\end{IEEEeqnarray*}

To end this Section, following Section 10.3 
in \cite{Wo}, let us determine the half-form correction for ${\mathcal P}_{0,\sigma}$.
Let $K_{D_\sigma}$ be the line bundle with trivializing frame 
% $$
% \alpha_\sigma = w^1\wedge \cdots \wedge w^n \wedge dy_1\wedge \cdots \wedge dy_r = (-1)^{r(r-1)/2 + (n-r)r}(\iota^*\omega)^r \wedge w^{r+1} \wedge \cdots \wedge w^n,$$
\begin{IEEEeqnarray*}{rCls}
    \alpha_\sigma & = & w^1\wedge \cdots \wedge w^n \wedge dy_1\wedge \cdots \wedge dy_r \\
    & = & (-1)^{r(r-1)/2 + (n-r)r}(\iota^*\omega)^r \wedge w^{r+1} \wedge \cdots \wedge w^n,
\end{IEEEeqnarray*}
that is the line bundle with fibers given by the space of $(n+r)$-forms which are annihilated by $D_\sigma$. Let $K_{{\mathcal P}_{0,\sigma}}$ be the line bundle of $n$-forms annihilating 
$\bar {\mathcal P}_{0,\sigma}$, with trivializing frame given by
$$\Omega_{0,\sigma}=dz^1_\sigma\wedge\cdots\wedge dz_\sigma^r \wedge w^{r+1} \wedge \cdots \wedge w^n.$$
 
\begin{lemma}We have
% $$
% \Omega_{0,\sigma} = e^{\sigma {\mathcal L}_{X_f}} w^1 \wedge \cdots \wedge w^n = dz^1_\sigma\wedge\cdots\wedge dz_\sigma^r \wedge w^{r+1} \wedge \cdots \wedge w^n
% $$
\begin{IEEEeqnarray*}{rCls}
    \Omega_{0,\sigma} & = & e^{\sigma {\mathcal L}_{X_f}} w^1 \wedge \cdots \wedge w^n \\
    & = & dz^1_\sigma\wedge\cdots\wedge dz_\sigma^r \wedge w^{r+1} \wedge \cdots \wedge w^n.
\end{IEEEeqnarray*}
\end{lemma} 

\begin{proof}
One verifies straighforwardly that, for $l=1,\dots r$,
% $$
% e^{\sigma{\mathcal L}_{X_f}} w^l = w^l + \sigma \sum_{k=1}^r F_{lk} dy^k = dz^l_\sigma,
% $$
\begin{IEEEeqnarray*}{rCls}
    e^{\sigma{\mathcal L}_{X_f}} w^l & = & w^l + \sigma \sum_{k=1}^r F_{lk} dy^k \\
    & = & dz^l_\sigma,
\end{IEEEeqnarray*}
while for $l=r+1,\dots , n$ we have
\begin{IEEEeqnarray*}{rCls+x*}
    e^{\sigma{\mathcal L}_{X_f}} w^{r+1}\wedge \cdots \wedge w^n & = & \det \left(e^{-\sigma {(\ad_{Fy})}_{\vert_{{\frak t}^\perp}}}\right)w^{r+1}\wedge \cdots \wedge w^n \\
    & = & w^{r+1}\wedge \cdots \wedge w^n. & & \hfill\qedhere
\end{IEEEeqnarray*}
% \begin{equation*}
%     e^{\sigma{\mathcal L}_{X_f}} w^{r+1}\wedge \cdots \wedge w^n = \det \left(e^{-\sigma {(\ad_{Fy})}_{\vert_{{\frak t}^\perp}}}\right)w^{r+1}\wedge \cdots \wedge w^n = w^{r+1}\wedge \cdots \wedge w^n. \qedhere
% \end{equation*}
% $$
%e^{\sigma{\mathcal L}_{X_f}} w^{r+1}\wedge \cdots \wedge w^n = \det \left(e^{-\sigma {(\ad_{Fy})}_{\vert_{{\frak t}^\perp}}}\right)w^{r+1}\wedge \cdots \wedge w^n = w^{r+1}\wedge \cdots \wedge w^n.
%$$
\end{proof}
 
We have
$$
\Omega_{0,\sigma} = \iota_{V_\sigma} \alpha_\sigma,
$$
where 
$$
V_\sigma = \left(\frac{2\sigma_2}{i}\right)^r (\det F) (-1)^{r^2+(n-r)r} \frac{\partial}{\partial \bar z_\sigma^1}\wedge \cdots \wedge \frac{\partial}{\partial \bar z_\sigma^r}. 
$$
The half-form measure for integration on $T^*K/D \cong G\times \{ i\frak t\}$ is given by (see Section 10.3 in \cite{Wo}),
$$
\vert\sqrt{\Omega_{0,\sigma}}\vert^2 = \left((-1)^{r(r-1)/2} \frac{\iota_{\bar V \wedge V}\omega^r}{r! (2\pi i)^r}\right)^\frac12 \alpha_\sigma.
$$
One then obtains, straighforwardly, 
\begin{proposition}\label{partialhalfform} We have 
$$
\vert\sqrt{\Omega_{0,\sigma}}\vert^2 = \pi^{-\frac{r}{2}} \sigma_2^{\frac{r}{2}} (\det F)^\frac12 \alpha_\sigma.
$$
\end{proposition}

\section{Quantum theory}

\subsection{Geometric quantization of $T^*G$}

Let $L\to T^*G$ be the trivial complex line bundle equipped with the standard Hermitian structure and with the connection
$$
\nabla = d + i\theta,
$$
whose curvature is $-i\omega.$  The Hilbert space of quantum states that is produced by geometric quantization of $(T^*G,\omega)$ in the polarization ${\mathcal P}$ is then, roughly, given by the space of sections of $L$ covariantly constant along $\overline{\mathcal P}$. However, one must also take into account $L^2$-conditions and the so-called half-form correction.

Let ${\mathcal P}_{0,0}$ be the vertical, or Schr\"odinger, polarization of $T^*G$ given by the kernel of the differential of the canonical projection 
$T^*G\cong G\times \g\to G$. In this case, the space of half-form corrected quantum states is \cite{Ha2,FMMN1,FMMN2}
$$
{\mathcal H}_{{\mathcal P}_{0,0}}=\{ f\otimes \sqrt{dx},\, f\in L^2(G,dx)\},
$$
where $dx$ stands for the Haar measure and $\sqrt{dx}$ denotes the half-form \cite{Wo}.

For the K\"ahler polarizations,  ${\mathcal P}_{\tau,0}$, $\tau \in \C^+$, given by the holomorphic tangent space of $(T^*G,J_{\tau,0})$, where $J_{\tau,0}$ are the complex structures in Proposition \ref{cstau}, one obtains \cite{Ha2,KMN1}

% \begin{eqnarray}\nonumber
% {\mathcal H}_{{\mathcal P}_{\tau,0}}=&&\{ F(x e^{\tau u_h(y)}) e^{i\tau (\langle u_h(y), y\rangle -h(y))} \otimes \sqrt{\Omega_{\tau,0}}\,\vert\, F\colon G_{\C}\to {\C} \,\,\mbox{is holomorphic},\,\,\\ \nonumber
% && \int_{G\times \g} \vert F\vert^2 \vert\Omega_{\tau,0}\vert^2 dx dy<\infty\}.
% \end{eqnarray}
\begin{IEEEeqnarray*}{rCls}
    {\mathcal H}_{{\mathcal P}_{\tau,0}} & = & \biggl\{ F(x e^{\tau u_h(y)}) e^{i\tau (\langle u_h(y), y\rangle -h(y))} \otimes \sqrt{\Omega_{\tau,0}} \biggm| F\colon G_{\C}\to {\C} \,\,\mbox{is holomorphic}, \\
    & & \hphantom{\biggl\{ F(x e^{\tau u_h(y)}) e^{i\tau (\langle u_h(y), y\rangle -h(y))} \otimes \sqrt{\Omega_{\tau,0}} \biggm| } \int_{G\times \g} \vert F\vert^2 \vert\Omega_{\tau,0}\vert^2 dx dy<\infty \biggr\}. \\
\end{IEEEeqnarray*}

Here, $\Omega_{\tau,0}$ is the pull-back by $\psi_\tau \circ \alpha_h$ of a non-vanishing (and therefore trivializing) 
left $G_{\C}$-invariant holomorphic section of the canonical bundle of $G_{\C}$ and $\sqrt{\Omega_{\tau,0}}$ denotes a choice of corresponding square root (see \cite{KMN1} for further discussion).

\subsection{Quantization with respect to ${\mathcal P}_{\tau,\sigma}$}

In this Section we consider $\tau \in {\C}^+, \sigma\in {\C}^+\cup {\R}$ and the K\"ahler polarizations ${\mathcal P}_{\tau,\sigma}$.

Let 
$$
{\mathcal H}_{{\mathcal P}_{\tau,\sigma}}= \overline{ \biggl\{s\otimes \sqrt{\Omega_{\tau,\sigma}} \biggm| s\in C^\infty(L),\ \nabla_{\overline{{\mathcal P}_{\tau,\sigma}}}s=0, \
\int_{T^*G} \bar s  s\vert\sqrt{\Omega_{\tau,\sigma}}\vert^2 \frac{\omega^n}{n!}<\infty \biggr\}}
$$
be the Hilbert space of ${\mathcal P}_{\tau,\sigma}$-polarized sections, where the closure is with respect to the inner product
$$
\langle s\otimes \sqrt{\Omega_{\tau,\sigma}}, s'\otimes \sqrt{\Omega_{\tau,\sigma}}\rangle_{\tau,\sigma}
=\int_{T^*G} \bar s s' \vert\sqrt{\Omega_{\tau,\sigma}}\vert^2 \frac{\omega^n}{n!}
.
$$ 
Recall that $C^\infty(L)=C^\infty(T^*G)\otimes {\C}$ since $L$ is the trivial line bundle.

\begin{theorem}
We have 
\begin{equation}\label{structure}
{\mathcal H}_{{\mathcal P}_{\tau,\sigma}}= \biggl\{ (\Phi\circ A_{\tau,\sigma}) e^{-i \lambda_{\tau,\sigma}}\otimes \sqrt{\Omega_{\tau,\sigma}} \biggm| \Phi\in {\mathcal O}(G_{\C}),\
\int_{T^*G} \vert \Phi\circ A_{\tau,\sigma}\vert^2 e^{-\kappa_{\tau,\sigma}}\vert\sqrt{\Omega_{\tau,\sigma}}\vert^2 \frac{\omega^n}{n!}<\infty \biggr\},
\end{equation}
where $\lambda_{\tau,\sigma}(x,y)= -\tau (\langle y, u_h(y)\rangle-h(y)) -\sigma f$.
\end{theorem}

\begin{proof}Let $\varphi \in C^\infty(T^*G)\otimes {\C}$ and suppose that 
$$
\nabla_{\bar E^{\tau,\sigma}_j}\, \left(\varphi e^{-i \lambda_{\tau,\sigma}}\right) =0.
$$
Since, from the proof of Theorem \ref{ohyes}, 
$$
\bar E^{\tau,\sigma}_j \lambda_{\tau,\sigma} = \theta(\bar E^{\tau,\sigma}_j), 
$$
we have
$$
(\bar E^{\tau,\sigma}_j + i\theta (\bar E^{\tau,\sigma}_j) )\left(\Phi e^{-i \lambda_{\tau,\sigma}}\right) =0 \Leftrightarrow \bar E^{\tau,\sigma}_j(\varphi)=0,
$$
so that $\varphi$ is $J_{\tau,\sigma}$-holomorphic, which, by Theorems \ref{sure} and \ref{super}, is equivalent to $\varphi = A_{\tau,\sigma}^* \Phi$, for some 
$\Phi\in {\mathcal O(G_{\C})}$. 
\end{proof}

Recall that one has the Peter-Weyl decomposition
\begin{equation}\label{peterweyl}
{\mathcal H}_{{\mathcal P}_{0,0}} = \overline{\oplus_{\lambda \in \hat G} V^{\lambda}_{0,0}},
\end{equation}
where $\hat G$ denotes the set of equivalence classes of irreducible representations fo $G$ and 
$$
V^{\lambda}_{0,0} = \bigl\{\pi^\lambda_{ij}(x)\otimes \sqrt{dx} \bigm| \lambda\in \hat G, i,j = 1,\dots, \dim \lambda \bigr\},
$$
where $\pi^\lambda_{ij}$ denotes the matrix elements for the irreducible representation with highest weight $\lambda$ and $\sqrt{\Omega_{0,0}}=\sqrt{dx}=\sqrt{w^1\wedge\cdots\wedge w^n}$ is the half-form correction for ${\mathcal P}_{0,0}$.

An holomorphic function $\Phi\in {\mathcal O}(G_{\C})$ is known to be given by an ``holomorphic Fourier series'' determined by the Peter-Weyl expansion of its restriction to $G$,
$$
\Phi (g) = \sum_{\lambda\in \hat G} \sum_{i,j=1}^{\dim \lambda}a_{ij}^\lambda \pi_{ij}^\lambda(g),\quad g\in G_{\C}, a_{ij}\in {\C},
$$ 
where $\pi^\lambda_{ij}$ also denote the matrix elements for the holomorphic representation of $G_{\C}$ with highest weight $\lambda$. 
(See Section 8 of \cite{Ha1}.)

\begin{proposition}\label{matrixelementisl2}
Let $\lambda\in \hat G$, $i,j=1,\dots,\dim \lambda$. Then,
$$
\pi^\lambda_{ij} \circ A_{\tau,\sigma} e^{-i\lambda_{\tau,\sigma}}\otimes \sqrt{\Omega_{\tau,\sigma}} \in {\mathcal H}_{{\mathcal P}_{\tau,\sigma}}.
$$
\end{proposition}

\begin{proof}It only remains to show that the $L^2$ condition in (\ref{structure}) is satisfied. Equation (\ref{kahlerpot}) ensures, as in the 
the proof of Theorem 4.6 in \cite{KMN1}, that the factor
$e^{-\kappa_{\tau,\sigma}}$ decays at least like a Gaussian along the imaginary directions in the Lie algebra. The proof of Thereom 4.6 in \cite{KMN1} 
and Corollary \ref{half-formgrowth} then ensure that the integral giving $\vert\vert \pi^\lambda_{ij} \circ A_{\tau,\sigma} e^{-i\lambda_{\tau,\sigma}}\otimes \sqrt{\Omega_{\tau,\sigma}}\vert\vert_{\tau,\sigma}^2$ is convergent.
\end{proof}

Therefore, we also have the decomposition
\begin{equation}\label{decomptausigma}
{\mathcal H}_{{\mathcal P}_{\tau,\sigma}} = \overline{\oplus_{\lambda \in \hat G} V^{\lambda}_{\tau,\sigma}},
\end{equation} 
where
\begin{equation*}
V^{\lambda}_{\tau,\sigma} = \bigl\{\pi^\lambda_{ij}(xe^{\tau u(y)}e^{\sigma Fy}) e^{-i \lambda_{\tau,\sigma}}\otimes \sqrt{\Omega_{\tau,\sigma}} \bigm|  i,j = 1,\dots, \dim \lambda\bigr\}.
\end{equation*}

\subsection{Quantization with respect to ${\mathcal P}_{0,\sigma}$}

In this Section we consider $\tau=0, \sigma_2>0$ and the mixed  polarizations ${\mathcal P}_{0,\sigma}$.

Let 
$$
{\mathcal H}_{{\mathcal P}_{0,\sigma}}= \overline{ \biggl\{s\otimes \sqrt{\Omega_{0,\sigma}} \biggm| s\in C^\infty(L),\ \nabla_{\overline{{\mathcal P}_{0,\sigma}}}s=0, \
\int_{G\times \{i\frak t\}} \bar s  s\vert\sqrt{\Omega_{0,\sigma}}\vert^2 <\infty \biggr\}}
$$
be the Hilbert space of ${\mathcal P}_{0,\sigma}$-polarized sections, where the closure is with respect to the inner product
$$
\langle s\otimes \sqrt{\Omega_{0,\sigma}}, s'\otimes \sqrt{\Omega_{0,\sigma}}\rangle_{0,\sigma}
=\int_{G\times \{i\frak t\}} \bar s s' \vert\sqrt{\Omega_{0,\sigma}}\vert^2,
$$ 
and where $\vert\sqrt{\Omega_{0,\sigma}}\vert^2$ is given in Proposition \ref{partialhalfform}.
Recall, again, that
\begin{equation*}
    C^\infty(L)=C^\infty(T^*G)\otimes {\C}
\end{equation*}
since $L$ is the trivial line bundle.

\begin{theorem}
We have 
\begin{equation}\label{structure2}
{\mathcal H}_{{\mathcal P}_{0,\sigma}}= \overline{ \biggl\{ \phi e^{-i \lambda_{0,\sigma}}\otimes \sqrt{\Omega_{0,\sigma}} \biggm| \phi\in C^\infty({{\mathcal P}_{0,\sigma}}),\
\int_{G\times i{\frak t}} \vert \phi \vert^2 e^{-\kappa_{0,\sigma}}\vert\sqrt{\Omega_{0,\sigma}}\vert^2 <\infty \biggr\}},
\end{equation}
where $\lambda_{0,\sigma}(x,y)= -\sigma f(y)$ and $C^\infty({{\mathcal P}_{0,\sigma}})$ stands for the space of ${\mathcal P}_{0,\sigma}$-polarized 
smooth functions on $T^*G$.
\end{theorem}

\begin{proof}Let $\phi \in C^\infty(T^*G)\otimes {\C}$ and suppose that 
$$
\nabla_{\bar E^{0,\sigma}_j}\, \left(\phi e^{-i \lambda_{0,\sigma}}\right) =0.
$$
Since, from the proof of Theorem \ref{ohyes}, 
$$
\bar E^{0,\sigma}_j \lambda_{0,\sigma} = \theta(\bar E^{0,\sigma}_j), 
$$
we have
$$
(\bar E^{0,\sigma}_j + i\theta (\bar E^{0,\sigma}_j) )\left(\phi e^{-i \lambda_{0,\sigma}}\right) =0 \Leftrightarrow \bar E^{0,\sigma}_j(\phi)=0.
$$
On the other hand, since $\theta(\frac{\partial}{\partial y_j})=0$, 
$$
\nabla_{\frac{\partial}{\partial y_j}} \,\phi e^{-i \lambda_{0,\sigma}} =0 \Leftrightarrow  \frac{\partial}{\partial y_j} \left(\phi e^{-i \lambda_{0,\sigma}}\right) =0 \Leftrightarrow  \frac{\partial}{\partial y_j} \left(\phi\right) =0,\quad j=r+1, \dots n,
$$
so that $\phi\in C^\infty({\mathcal P}_{0,\sigma})$.
\end{proof}

\begin{proposition}Let $\phi \in C^\infty({\mathcal P}_{0,\sigma})$. Then, $\phi$ has an expansion
$$
\phi(x,y) = \sum_{\lambda\in \hat G} \sum_{i,j=1}^{\dim\,\lambda}a_{ij}^\lambda \pi^\lambda_{ij} (x e^{\sigma Fy}).
$$
\end{proposition}

\begin{proof}
>From Proposition \ref{partialpullback}, we see that $\phi \in C^\infty({\mathcal P}_{0,\sigma})$, being $J_{\sigma,0}^{L_{(x,y)}}$-holomorphic along the leaf 
$L_{(x,y)}$, must be given along such leaves by the pull-back of an holomorphic function on $T_{\C}$. In more detail, let $U$ be a sufficiently small open set on the flag manifold $G/T$ and let $s_U$ be a local section of the principal fiber bundle $p\colon  G \to G/T$. We then have a diffeomorphism
\begin{eqnarray}\nonumber
p^{-1}(U)\times {\frak g} = p^{-1}(U)\times {\frak t} \times {\frak t}^\perp &\stackrel{\alpha_U}{\to}& U\times T_{{\C}} \times   {\frak t}^\perp\\ \nonumber
(x,y)  &\mapsto& (p(x), t_x e^{\sigma Fy}, y^\perp ),
\end{eqnarray}
where $x= s_U(p(x))\cdot t_x$ and $p^{-1}(U)\times {\frak g}\subset T^*G$ is open. From Proposition \ref{partialpullback} and from the Fourier expansion of holomorphic functions on $T_{\C}$ , it is then clear that, over 
$p^{-1}(U)\times {\frak g}$, $\phi\in C^\infty({\mathcal P}_{0,\sigma})$ is of the form 
$$
\phi = \alpha_U^*  \left(\sum_{\lambda\in \hat T} a_\lambda^s (p(x)) e^{\lambda}(t_{x}e^{\sigma Fy})\right),
$$ 
where $\hat T$ is the set of characters for $T$, which we identify with the set of weights of $G$. If $s'_U = t_0^{-1} \cdot s_U, t_0\in T$, is another local section 
then,
$$
a_\lambda^{s} (p(x) ) = a_\lambda^{s'} (p(x)) e^{\lambda}(t_0),
$$
so that 
$$
a_\lambda^s (p(x)) e^\lambda (t_x)
$$
is a globally defined smooth function on $G$ which is $T$-equivariant with weight $\lambda$. It follows, from the Peter-Weyl expansion of smooth functions on $G$, that
this function can be expanded in a series of matrix elements  $\{\pi^{\tilde \lambda}_{ij}\},\,\, \tilde \lambda\in \hat G$, where the only contribution comes from representations $\tilde \lambda$ for which $\lambda$ is a weight. Writing these matrix representatives $\pi^{\tilde \lambda}$ in a basis of weight spaces, we
immediately obtain the statement of the Proposition. 
\end{proof}

\begin{proposition}\label{peterweylpartial}Let $\lambda \in \hat G$. Then,
$$
\pi^\lambda_{ij} (x e^{\sigma Fy}) e^{-i\lambda_{0,\sigma}} \otimes \sqrt{\Omega_{0,\sigma}} \in {\mathcal H}_{{\mathcal P}_{0,\sigma}}.
$$
\end{proposition}

\begin{proof}This is similar to the proof of Proposition \ref{matrixelementisl2} where now the Gaussian factor $e^{-\kappa_{0,\sigma}}$ ensures convergence of the integral along the non-compact factor $i\frak t$.
\end{proof}

\begin{corollary}We have the decomposition
    \begin{equation*}
        {\mathcal H}_{{\mathcal P}_{0,\sigma}} = \overline{\oplus_{\lambda \in \hat G} V^\lambda_{0,\sigma}},
    \end{equation*}
where
$$
V^\lambda_{0,\sigma} = \bigl\{\pi^\lambda_{ij}(xe^{\sigma Fy})e^{-i\lambda_{0,\sigma}}\otimes \sqrt{\Omega_{0,\sigma}} \bigm| i,j =1, \dots, \dim \lambda \bigr\}.
$$
\end{corollary}

\section{Partial coherent state transforms and unitarity}

\subsection{The generalized coherent state transforms $U_{\tau,\sigma}$}

Recall from \cite{KMN1} that there is a natural $G\times G$-action on ${\mathcal H}_{{\mathcal P}_{\tau,0}},$ for $\tau \in {\C}^+$, which 
extends the $G\times G$-action on ${\mathcal H}_{{\mathcal P}_{0,0}}$ associated to the Peter-Weyl decomposition (\ref{peterweyl}). One has, for
$x', \tilde x\in G$,
$$
(x',\tilde x)\cdot \Phi \bigl(x{e^{\tau u_h(y)}}\bigr)e^{-i\lambda_{\tau,0}}\otimes \sqrt{\Omega_{\tau,0}} = \Phi \bigl(x'x{e^{\tau u_h(y)}}\tilde x\bigr)e^{-i\lambda_{\tau,0}}\otimes \sqrt{\Omega_{\tau,0}},
$$
for $\Phi \in {\mathcal O}(G_{\C})$. Note that $\lambda_{\tau,0}$, in particular, is $G\times G$-invariant. This action preserves the decomposition in (\ref{decomptausigma}). In the case when $\sigma\neq 0$, that we consider in this paper, one obtains a $G\times T$-action instead. In particular, note that 
$\lambda_{\tau,\sigma}$ is only $G\times T$-invariant in this case since, in general, $\exp (\sigma F \Ad_{\tilde x^{-1}}y) \neq \Ad_{\tilde x^{-1}} 
(\exp (\sigma Fy))$.

Therefore, we will consider a finer decomposition, for $\tau \in {\C}^+, \sigma \in {\C}^+\cup {\R}$,
\begin{equation}\label{finer}
{\mathcal H}_{{\mathcal P}_{\tau,\sigma}} = \bigoplus_{\lambda \in \hat G} \oplus_{j=1}^{\dim \lambda} V_{\tau,\sigma}^{\lambda,\lambda_j},
\end{equation} 
where $\{\lambda_j\}_{j=1, \dots, \dim \lambda}$ is the set of weights of the irreducible representation of highest weight $\lambda$, with $\lambda_1 = \lambda$, and
$$
V_{\tau,\sigma}^{\lambda,\lambda_j} = \spn_{\C} \bigl\{ \pi^\lambda_{kj}\bigl(x e^{\tau u_h(y)}e^{\sigma Fy}\bigr) e^{-i\lambda_{\tau,\sigma}} \otimes \sqrt{\Omega_{\tau,\sigma}} \bigm| k=1,\dots, \dim \lambda \bigr\},
$$
with $\pi^\lambda$ written in the basis of weight vectors, so that for $a\in \frak t$,
$$
\pi^\lambda_{ki}\bigl(e^{a}\bigr) = {\rm Diag}\,\, \left(e^{i\langle \lambda_1, a\rangle}, \ldots, e^{i\langle \lambda_{\dim \lambda}, a\rangle} \right).
$$

We then have the natural action of $G\times T$ on ${\mathcal H}_{{\mathcal P}_{\tau,\sigma}}$,
$$
(x',t)\cdot \Phi \bigl(x e^{\tau u_h(y)}e^{\sigma Fy}\bigr) e^{-i\lambda_{\tau,\sigma}} \otimes \sqrt{\Omega_{\tau,\sigma}} = \Phi \bigl(x'x e^{\tau u_h(y)}e^{\sigma Fy}t\bigr) e^{-i\lambda_{\tau,\sigma}} \otimes \sqrt{\Omega_{\tau,\sigma}},
$$ 
for $\Phi \in {\mathcal O}(G_{\C})$ and $(x',t) \in G\times T$, which preserves the decomposition in (\ref{finer}). 

Let now $h_{pq}$ and $f_{pq}$ be the Kostant-Souriau prequantum operators, on the half-form corrected prequantum (trivial) bundle $L\otimes \sqrt{K_{{\mathcal P}_{\tau,\sigma}}}$, associated to $h$ and $f$,
\begin{IEEEeqnarray*}{rCl}\nonumber
h_{pq} &= &\left(i\nabla_{X_h} + h\right)\otimes 1 + 1 \otimes i{\mathcal L}_{X_h}\\ \nonumber
f_{pq}&= &\left(i\nabla_{X_f} + f\right)\otimes 1 + 1 \otimes i{\mathcal L}_{X_f}
\end{IEEEeqnarray*}

\begin{lemma}One has,
\begin{IEEEeqnarray}{rCl}\label{pqh}
h_{pq} &=& iX_h + (h-\langle y, u_h(y)\rangle), \\ \label{pqf}
f_{pq} &=& iX_f -f.
\end{IEEEeqnarray}
Moreover, as operators on $C^\infty(L\otimes \sqrt{K_{{\mathcal P}_{\tau,\sigma}}})$,
\begin{equation}\label{pqscommute}
[h_{pq},f_{pq}]=0.
\end{equation}
\end{lemma}

\begin{proof}(\ref{pqh}) and (\ref{pqf}) follow directly from the definition, while (\ref{pqscommute}) is just a restatement of the fact that 
$[X_h,X_f]=0$.
\end{proof}

Following the ideology of \cite{KMN1,KMN2}, inspired by the structure of the coherent state transforms of Hall \cite{Ha1, Ha2}, we now introduce the quantum 
operators on ${\mathcal H}_{{\mathcal P}_{0,0}}$ in view of (\ref{finer}). 
Let $\rho\in {\frak t}$ be the Weyl vector defined by half the sum of the positive roots of 
${\frak g}\otimes {\C}$. Define
% \begin{eqnarray}\label{quantumh}\nonumber
% Q(h)&\colon & {\mathcal H}_{{\mathcal P}_{0,0}} \to {\mathcal H}_{{\mathcal P}_{0,0}}\\
% \pi_{ij}^\lambda(x) \times \sqrt{dx}&\mapsto& h(-(\lambda+\rho)) \pi_{ij}^\lambda(x)\otimes \sqrt{dx}
% \end{eqnarray}
% and 
% \begin{eqnarray}\label{quantumf}\nonumber
% Q(f)&\colon & {\mathcal H}_{{\mathcal P}_{0,0}} \to {\mathcal H}_{{\mathcal P}_{0,0}}\\
% \pi_{ij}^\lambda(x) \otimes\sqrt{dx}&\mapsto& f(-\lambda_j) \pi_{ij}^\lambda(x)\otimes \sqrt{dx}.
% \end{eqnarray}
% \begin{IEEEeqnarray}{rrCl}\label{quantumh}
%     Q(h) \colon & {\mathcal H}_{{\mathcal P}_{0,0}} & \to & {\mathcal H}_{{\mathcal P}_{0,0}} \\
%     & \pi_{ij}^\lambda(x) \times \sqrt{dx} & \mapsto & h(-(\lambda+\rho)) \pi_{ij}^\lambda(x)\otimes \sqrt{dx} \nonumber
% \end{IEEEeqnarray}
% and
% \begin{IEEEeqnarray}{rrCl}\label{quantumf}
%     Q(f) \colon & {\mathcal H}_{{\mathcal P}_{0,0}} & \to & {\mathcal H}_{{\mathcal P}_{0,0}} \\
%     & \pi_{ij}^\lambda(x) \times \sqrt{dx} & \mapsto & f(-\lambda_j) \pi_{ij}^\lambda(x)\otimes \sqrt{dx} \nonumber
% \end{IEEEeqnarray}
\begin{IEEEeqnarray}{rrCl}\label{quantumh}
    Q(h) \colon & {\mathcal H}_{{\mathcal P}_{0,0}} & \to & {\mathcal H}_{{\mathcal P}_{0,0}} \\
    & \pi_{jk}^\lambda(x) \otimes \sqrt{dx} & \mapsto & h(-(\lambda+\rho)) \pi_{jk}^\lambda(x)\otimes \sqrt{dx} \nonumber \\
    \noalign{\noindent and \vspace{2\jot}}
    Q(f) \colon & {\mathcal H}_{{\mathcal P}_{0,0}} & \to & {\mathcal H}_{{\mathcal P}_{0,0}} \label{quantumf} \\
    & \pi_{jk}^\lambda(x) \otimes \sqrt{dx} & \mapsto & f(-\lambda_j) \pi_{jk}^\lambda(x)\otimes \sqrt{dx} \nonumber
\end{IEEEeqnarray}

Obviously, these operators commute with each other and they preserve the decomposition in (\ref{finer}).

\begin{lemma}\label{ajax}Let $\lambda\in \hat G, i,j=1,\dots, \dim \lambda.$ Then,
$$
e^{-i\tau h_{pq}} \circ e^{-i\sigma f_{pq}} \pi^\lambda_{jk}(x)\otimes \sqrt{dx} = e^{-i\lambda_{\tau,\sigma}} \pi^\lambda_{jk} \bigl(xe^{\tau u(y)}e^{\sigma Fy}\bigr)\otimes \sqrt{\Omega_{\tau,\sigma}}.
$$
\end{lemma}

\begin{proof}The fact that
$$
e^{\tau {\mathcal L}_{X_h}}\circ e^{\sigma {\mathcal L}_{X_f}}\Omega_{0,0} =  e^{\sigma {\mathcal L}_{X_f}} \circ e^{\tau {\mathcal L}_{X_h}} \Omega_{0,0}= \Omega_{\tau,\sigma},
$$
follows from \cite{MN}, (\ref{sectioncanonical}) and Theorem 3.10 in \cite{KMN1}. Then (\ref{pqh}), (\ref{pqf}), \cite{MN}, Theorem \ref{super} and the proof of Theorem 3.7 in \cite{KMN1} give the statement of the Lemma.
\end{proof}

We obtain the following
\begin{corollary}The operator 
$$
e^{-i\tau h_{pq}} \circ e^{-i\sigma f_{pq}}
$$
is a densely defined operator from ${\mathcal H}_{0,0}$ to ${\mathcal H}_{\tau,\sigma}$.
\end{corollary}

Let us now define the generalized CST,
\begin{equation*}
U_{\tau,\sigma}\colon  {\mathcal H}_{0,0}\to {\mathcal H}_{\tau,\sigma},
\end{equation*}
with
\begin{equation*}
%\label{csttausigma}
U_{\tau,\sigma} = e^{-i\tau h_{pq}} \circ e^{-i\sigma f_{pq}} \circ e^{i\tau Q(h)} \circ e^{i\sigma Q(f)},
\end{equation*}
for $\tau \in {\C}^+, \sigma \in {\C}^+\cup {\R}$.

The above implies the following
\begin{theorem}Let $\tau \in {\C}^+, \sigma \in {\C}^+\cup {\R}$. Then, the generalized CST $U_{\tau, \sigma}$ is a linear isomorphism that intertwines the $G\times T$-actions on 
${\mathcal H}_{0,0}$ and ${\mathcal H}_{\tau,\sigma}$.
\end{theorem}

\begin{remark}
Note that in the case $\tau\in {\C}^+, \sigma =0$, the generalized CST $U_{\tau,0}$ intertwines the full $G\times G$ actions on 
${\mathcal H}_{{\mathcal P}_{0,0}}$ and on ${\mathcal H}_{{\mathcal P}_{\tau,0}}$. (See\cite{KMN1, KMN2}.) Therefore, while the CSTs $U_{\tau,0}$ are 
``$G\times G$-invariant'', for $\sigma\neq 0$ the CSTs $U_{\tau,\sigma}$ are only ``$G\times T$-invariant''.
\end{remark}

Setting $\tau =0$ in Lemma \ref{ajax} we obtain
$$
e^{-i\sigma f_{pq}} \bigl(\pi^\lambda_{jk}(x)\otimes \sqrt{dx}\bigr) = e^{-i\lambda_{0,\sigma}} \pi^\lambda_{jk}\bigl(xe^{\sigma Fy}\bigr) \otimes \sqrt{\Omega_{0,\sigma}}
$$
and the following

\begin{corollary} The operator 
$$
e^{-i\sigma f_{pq}}
$$
is a densely defined operator from ${\mathcal H}_{0,0}$ to ${\mathcal H}_{0,\sigma}$ which preserves the decomposition in (\ref{finer}).
\end{corollary}

We now define the partial CST,
\begin{equation*}
U_{0,\sigma}\colon  {\mathcal H}_{0,0} \to {\mathcal H}_{0,\sigma},
\end{equation*}
by
\begin{equation*}
U_{0,\sigma} = e^{-i\sigma f_{pq}} \circ e^{i\sigma Q(f)},
\end{equation*}
for $\sigma \in {\C}^+.$ From the above, one obtains the following 
\begin{theorem}Let $\sigma \in {\C}^+$. The partial CST $U_{0,\sigma}$ is a linear isomorphism that intertwines the $G\times T$-actions on ${\mathcal H}_{0,0}$ and on ${\mathcal H}_{0,\sigma}$.
\end{theorem}

\subsection{Unitarity of the partial coherent state transform $U_{0,\sigma}$}

In this Section, we will establish that, in fact, $U_{0,\sigma}$ is a unitary isomorphism of Hilbert spaces.

\begin{theorem}Let $\sigma \in {\C}^+$. The partial coherent state transform 
$$
U_{0,\sigma}\colon  {\mathcal H}_{{\mathcal P}_{0,0}} \to {\mathcal H}_{{\mathcal P}_{0,\sigma}}
$$
is a unitary isomorphism.
\end{theorem}

\begin{proof}
>From (\ref{structure2}), (\ref{finer}), Proposition \ref{partialhalfform}, Lemma \ref{ajax} and (\ref{quantumf}) and we want to compute
% \begin{eqnarray*}
% &&\langle \pi^\lambda_{ij}(xe^{\sigma Fy}) e^{-i\lambda_{0,\sigma}} \otimes \sqrt{\Omega_{0,\sigma}}, \pi^\lambda_{kl}(xe^{\sigma Fy}) e^{-i\lambda_{0,\sigma}} \otimes \sqrt{\Omega_{0,\sigma}}\rangle_{{\mathcal H}_{{\mathcal P}_{0,\sigma}}} =\\
% && \int_{G\times \frak t} \overline{\pi^\lambda_{ij}(xe^{\sigma Fy})} \pi^{\lambda'}_{kl}(xe^{\sigma Fy}) e^{i\overline \lambda_{0,\sigma}}e^{-i\lambda_{0,\sigma}}
% \overline{e^{i\sigma f(-\lambda j)}} e^{i\sigma f(-\lambda_l)} \vert\vert\sqrt{\Omega_{0,\sigma}}\vert\vert^2 = \\
% &=& \pi^{-\frac{r}{2}} \sigma_2^{\frac{r}{2}} (\det F)^\frac12 \overline{e^{i\sigma f(-\lambda j)}} e^{i\sigma f(-\lambda_l)} \\
% &&\int_{G\times \frak t}  
% \left(\sum_{s=1}^{\dim \lambda} \pi^\lambda_{si}(e^{-\sigma_1 Fy}x^{-1}) \pi^\lambda_{js}(e^{i\sigma_2 Fy})\right)  \pi^{\lambda'}_{kl}(xe^{\sigma Fy})  e^{-\sigma_2\langle y, Fy\rangle}dx dy. 
     %   \end{eqnarray*}
\begin{IEEEeqnarray*}{rCls}
    \IEEEeqnarraymulticol{3}{l}{\bigl\langle \pi^\lambda_{jk}\bigl(xe^{\sigma Fy}\bigr) e^{-i\lambda_{0,\sigma}} \otimes \sqrt{\Omega_{0,\sigma}}, \pi^\lambda_{lm}\bigl(xe^{\sigma Fy}\bigr) e^{-i\lambda_{0,\sigma}} \otimes \sqrt{\Omega_{0,\sigma}}\bigr\rangle_{{\mathcal H}_{{\mathcal P}_{0,\sigma}}}}\\
    \quad & = & \int_{G\times \frak t} \overline{\pi^\lambda_{jk}\bigl(xe^{\sigma Fy}\bigr)} \pi^{\lambda'}_{lm}\bigl(xe^{\sigma Fy}\bigr) e^{i\overline \lambda_{0,\sigma}}e^{-i\lambda_{0,\sigma}} \overline{e^{i\sigma f(-\lambda_k)}} e^{i\sigma f(-\lambda_m)} \vert\vert\sqrt{\Omega_{0,\sigma}}\vert\vert^2\\
    \quad & = & \pi^{-\frac{r}{2}} \sigma_2^{\frac{r}{2}} (\det F)^\frac12 \overline{e^{i\sigma f(-\lambda_k)}} e^{i\sigma f(-\lambda_m)} \times \\
    \quad & & {}\times \int_{G\times \frak t}  \left(\sum_{s=1}^{\dim \lambda} \pi^\lambda_{sj}\bigl(e^{-\sigma_1 Fy}x^{-1}\bigr) \pi^\lambda_{ks}\bigl(e^{i\sigma_2 Fy}\bigr)\right)  \pi^{\lambda'}_{lm}\bigl(xe^{\sigma Fy}\bigr)  e^{-\sigma_2\langle y, Fy\rangle}dx dy.  \\
\end{IEEEeqnarray*}
By  the Weyl orthogonality relations, integration on $G$ gives then
\begin{IEEEeqnarray*}{rCl+x*}
\IEEEeqnarraymulticol{3}{l}{\bigl\langle \pi^\lambda_{jk}\bigl(xe^{\sigma Fy}\bigr) e^{-i\lambda_{0,\sigma}} \otimes \sqrt{\Omega_{0,\sigma}}, \pi^\lambda_{lm}\bigl(xe^{\sigma Fy}\bigr) e^{-i\lambda_{0,\sigma}} \otimes \sqrt{\Omega_{0,\sigma}}\bigr\rangle_{{\mathcal H}_{{\mathcal P}_{0,\sigma}}}} \\
\quad&=&\delta_{\lambda \lambda'} \delta_{jl}(\dim \lambda)^{-1} \pi^{-\frac{r}{2}} \sigma_2^{\frac{r}{2}} (\det F)^\frac12 \overline{e^{i\sigma f(-\lambda_k)}} e^{i\sigma f(-\lambda_m)} \int_{\frak t} \pi^\lambda_{kl}\bigl(e^{2\sigma_2 Fy}\bigr) e^{-\sigma_2\langle y, Fy\rangle} dy\\
\quad&=& \delta_{\lambda \lambda'} \delta_{jl}\delta_{km}(\dim \lambda)^{-1} \pi^{-\frac{r}{2}} \sigma_2^{\frac{r}{2}} (\det F)^\frac12  e^{-2\sigma_2\sigma f(-\lambda_k)} \int_{\frak t} e^{2\sigma_2 \langle \lambda_k, Fy\rangle} e^{-\sigma_2\langle y, Fy\rangle} dy\\
\quad&=& \delta_{\lambda \lambda'} \delta_{jl}\delta_{km}(\dim \lambda)^{-1} \\
\quad&=& \bigl\langle \pi^\lambda_{jk}(x) \sqrt{dx}, \pi^\lambda_{lm}(x) \otimes \sqrt{dx}\bigr\rangle_{{\mathcal H}_{{\mathcal P}_{0,0}}}.& \hfill\qedhere
\end{IEEEeqnarray*}
\end{proof}

\begin{remark}While for general $h$ the transform $U_{\tau,\sigma}$ will not be unitary (see, for example, \cite{KMN2}), for quadratic $h$ this seems to be a reasonable expectation. However, the evaluation of $\vert\vert U_{\tau,\sigma} \pi^\lambda_{jk}(x)\otimes \sqrt{dx}\vert\vert^2$ does not seem to be that straightforward.  
\end{remark}

\section{Appendix}
Here we collect two useful results.

\begin{notation}
	Let $S,T \colon \g \longrightarrow \g$ be differentiable maps (if $S,T \in \g$, then regard them as the constant map that to every $y \in \g$ assigns $S$ or $T$, respectively). Denote by $X^{S,T} \in \mathfrak{X}(T^*G)$ the left-invariant vector field defined by $S$ and $T$:
	\begin{IEEEeqnarray*}{RL}
		X^{S,T} \colon G \times \g & \longrightarrow \g \oplus \g \\
		(x,y) & \longmapsto (S(y),T(y)). 
	\end{IEEEeqnarray*}
\end{notation}

\begin{lemma}\label{lema1apen}
	Let $X^{A,B}, X^{C,D} \in \mathfrak{X}(T^*G)$ be left-invariant vector fields, where $A$, $B$, $C$, $D \colon \g \longrightarrow \g$ are differentiable maps. 
	The Lie bracket $\left[X^{A,B}, X^{C,D}\right]$ is given by
	\begin{IEEEeqnarray}{rCl}
		\label{eq:liebracket}
		\left[X^{A,B}, X^{C,D}\right]_{T^*G} (x,y) & = & \Bigl([A(y), C(y)]_\g , 0 \Bigr) \\
		& & {}+ \Bigl( (dC)|_y (B(y)) - (dA)|_y (D(y)), 0 \Bigr) \nonumber \\
		& & {}+ \Bigl(0, (dD)|_y (B(y)) - (dB)|_y (D(y))\Bigr). \nonumber
	\end{IEEEeqnarray}
\end{lemma}
\begin{proof}
	Let $\{T_1,...,T_n\}$ be a basis of $\g$ with associated coordinates $\{y^1,...,y^n\}$, and $\{X_1,...,X_n\}$ be the left invariant vector fields in $G$ such that $X_j|_e = T_j$. Then,
	\begin{equation*}
		X_1,...,X_n,\frac{\partial}{\partial y^1},...,\frac{\partial}{\partial y^n} \in \mathfrak{X}(G\times \g)
	\end{equation*}
	form a basis $TG$. Notice that as vector fields in $G \times \g$,
	\begin{IEEEeqnarray*}{rCl}
		\left[X_j,X_k\right] & = & \sum_{l=1}^{n} c_{jk}^l X_l,\\
		\left[X_j,\frac{\partial}{\partial y^k}\right] & = & 0,\\
		\left[\frac{\partial}{\partial y^j},\frac{\partial}{\partial y^k}\right] & = & 0,
	\end{IEEEeqnarray*}
	where $c_{jk}^l$ are the structure constants associated to the basis $\{T_1,...,T_n\}$. As vector fields on $G \times \g$,
	\begin{IEEEeqnarray*}{rCl}
		X^{A,B}(x,y) & = & \sum_{j=1}^{n} \left( a^j(y) X_j + b^j(y) \frac{\partial}{\partial y^j} \right) \\
		X^{C,D}(x,y) & = & \sum_{j=1}^{n} \left( c^j(y) X_j + d^j(y) \frac{\partial}{\partial y^j} \right)
	\end{IEEEeqnarray*}
	Using this new basis we can compute $\left[X^{A,B}, X^{C,D}\right]$:
	\begin{IEEEeqnarray*}{rCl}
		[X^{A,B},X^{C,D}]|_{(x,y)} & = & \sum_{j,k=1}^{n} \left[ a^j(y) X_j + b^j(y) \frac{\partial}{\partial y^j}, c^k(y) X_k + d^k(y) \frac{\partial}{\partial y^k} \right] \\
		& = & \underbrace{\sum_{j,k=1}^{n} \left[ a^j(y) X_j, c^k(y) X_k\right]}_{\Sigma_1} + \underbrace{\sum_{j,k=1}^{n} \left[ a^j(y) X_j, d^k(y) \frac{\partial}{\partial y^k} \right]}_{\Sigma_2} \\
		& & {}+ \underbrace{\sum_{j,k=1}^{n} \left[ b^j(y) \frac{\partial}{\partial y^j}, c^k(y) X_k \right]}_{\Sigma_3} + \underbrace{\sum_{j,k=1}^{n} \left[ b^j(y) \frac{\partial}{\partial y^j}, d^k(y) \frac{\partial}{\partial y^k} \right]}_{\Sigma_4} \nonumber
	\end{IEEEeqnarray*}
	\begin{IEEEeqnarray*}{rCl}
		\Sigma_1 & = & \sum_{j,k,l=1}^n a^j(y) c^k(y) C_{jk}^l X_l = \sum_{j=1}^{n }\left( [A(y),C(y)] \right)^j X_j,\\
		\Sigma_2 & = & -\sum_{j,k=1}^{n} d^k(y) \frac{\partial a^j}{\partial y^k} X_j = \sum_{j=1}^{n } \left(-dA|_y(D(y))\right)^j X_j,\\
		\Sigma_3 & = & \sum_{j,k=1}^{n} b^k(y) \frac{\partial c^j}{\partial y^k} X_j = \sum_{j=1}^{n } \left(dC|_y(B(y))\right)^j X_j,\\
		\Sigma_4 & = & \sum_{j,k=1}^{n} \left( b^j(y) \frac{\partial d^k}{\partial y^j} \frac{\partial}{\partial y^k} - d^j(y) \frac{\partial b^k}{\partial y^j} \frac{\partial}{\partial y^k} \right) = \sum_{j=1}^{n } (D|_y(B(y)) - B|_y(D(y)))^j \frac{\partial}{\partial y^j},
	\end{IEEEeqnarray*}
	which proves the result.
\end{proof}

	\begin{lemma}
\label{lema2apen}Let $X_h, X_f$ be the Hamiltonian vector fields considered in Sections \ref{prel} \and \ref{newpol}.
	We have, for $t\in {\R}$,
	\begin{itemize}
		\item if $S,T \colon \g \longrightarrow \g$,
		\begin{equation}
		\label{eq:expLXhX}
			e^{t {\mathcal L}_{X_h}} \cdot X^{S,T} |_{(x,y)} = \left( e^{t \ad_{u_h(y)}} (S(y)) + \frac{1 - e^{t \ad_{u_h(y)}}}{\ad_{u_h(y)}} H_h(y) (T(y)), T(y) \right);
		\end{equation}
		\item if $S,T \in \g$,
		\begin{equation}
		\label{eq:expLXfX}
			e^{s {\mathcal L}_{X_f}} \cdot X^{S,T} |_{(x,y)} = \left(e^{s \, \ad_{Fy}} (S) -s F (T), \bigg[ e^{s \, \ad_{Fy}} -s  \ad_y \circ F \bigg] (T) \right);
		\end{equation}
		\item if $S,T \in \g$,
		\begin{multline}
		\label{eq:expLXhLXfX}
			e^{t {\mathcal L}_{X_h}+s {\mathcal L}_{X_f}} \cdot X^{S,T}|_{(x,y)} = \\ = \bigg( \bigg[ e^{t \ad_{u_h(y)}} e^{s \, \ad_{Fy}} \bigg](S) +  \bigg[\frac{\id - 
			e^{t \ad_{u_h(y)}}}{\ad_{u_h(y)}} H_h(y) e^{s \, \ad_{Fy}} -s F \bigg](T), \bigg[ e^{s \, \ad_{Fy}} -s  \ad_y \circ F \bigg] (T) \bigg).
		\end{multline}
	\end{itemize}
\end{lemma}
\begin{proof}
	Note that $e^{t {\mathcal L}_{X_h}} \cdot X^{S,T}$ is the unique one-parameter family of vector fields $X(t) \in \mathfrak{X}(M) \otimes \mathbb{C}$ such that
	\begin{equation*}
        	\left\{
	\begin{IEEEeqnarraybox}[\IEEEeqnarraystrutmode][c]{rCl}
	{\mathcal L}_{X_h} X(t) & = & \frac{d}{dt} X(t)\\
	X(0) & = & X^{S,T}
	\end{IEEEeqnarraybox}
	\right.
    \end{equation*}

	It is easily seen that
	\begin{multline*}
		\frac{d}{dt} \left( e^{t \ad_{u_h(y)}} (S(y)) + \frac{1 - e^{t \ad_{u_h(y)}}}{\ad_{u_h(y)}} H_h(y) (T(y)), T(y) \right)\\ = \bigg( \ad_{u_h(y)} e^{t \ad_{u_h(y)}} S(y) - e^{t \ad_{u_h(y)}} H_h(y) T(y), 0 \bigg).
	\end{multline*}
	And using equation (\ref{eq:liebracket}), it is possible to prove that
	\begin{multline*}
		{\mathcal L}_{X_h} \left( e^{t \ad_{u_h(y)}} (S(y)) + \frac{1 - e^{t \ad_{u_h(y)}}}{\ad_{u_h(y)}} H_h(y) (T(y)), T(y) \right) \\= \bigg( \ad_{u_h(y)} e^{t 
		\ad_{u_h(y)}} S(y) - e^{t \ad_{u_h(y)}} H_h(y) T(y), 0 \bigg).
	\end{multline*}
	Therefore, equation (\ref{eq:expLXhX}) is proved. Using the same reasoning, it is easily seen that
	\begin{multline*}
		\frac{d}{ds} \bigg( e^{s \, \ad_{Fy}} S - sFT, \left[ e^{s \, \ad_{Fy}} -s \, \ad_y \circ F \right] T \bigg) \\=\bigg( \ad_{Fy} e^{s \, \ad_{Fy}} S - FT, \left[ \ad_{Fy}e^{s \, \ad_{Fy}} - \ad_y \circ F \right] T \bigg).
	\end{multline*}
	Again, using equation (\ref{eq:liebracket}), and the properties of $F$, it is possible to prove that
	\begin{multline*}
	{\mathcal L}_{X_f} \bigg( e^{s \, \ad_{Fy}} S - sFT, \left[ e^{s \, \ad_{Fy}} - s \, \ad_y \circ F \right] T \bigg) \\=\bigg( \ad_{Fy} e^{s \, \ad_{Fy}} S - FT, \left[ \ad_{Fy}e^{s \, \ad_{Fy}} - \ad_y \circ F \right] T \bigg),
	\end{multline*}
	which proves (\ref{eq:expLXfX}). To prove (\ref{eq:expLXhLXfX}), it suffices to use the fact that ${\mathcal L}_{X_h} {\mathcal L}_{X_f} = {\mathcal L}_{X_f} 
	{\mathcal L}_{X_h}$, and equations (\ref{eq:liebracket}) and (\ref{eq:expLXfX}):
	\begin{IEEEeqnarray*}{rCl+x*}
        \IEEEeqnarraymulticol{3}{l}{e^{t {\mathcal L}_{X_h}+s {\mathcal L}_{X_f}} \cdot X^{S,T}|_{(x,y)}}\\
        \quad & = & e^{t {\mathcal L}_{X_h}} \left( e^{s {\mathcal L}_{X_f}} \cdot X^{S,T} \right) |_{(x,y)} \\
		\quad& = & e^{t {\mathcal L}_{X_h}} \left(e^{s \, \ad_{Fy}} (S) -s F (T), \bigg[ e^{s \, \ad_{Fy}} -s  \ad_y \circ F \bigg] (T) \right) \\
		\quad& = & \bigg( e^{t \ad_{u_h(y)}} (e^{s \, \ad_{Fy}} (S) -s F (T)) + \frac{1 - e^{t \ad_{u_h(y)}}}{\ad_{u(y)}} H_h(y) \bigg(\bigg[ e^{s \, \ad_{Fy}} -s \, \ad_y \circ F \bigg] (T)\bigg), \nonumber \\
        \IEEEeqnarraymulticol{3}{r}{\bigg[ e^{s \, \ad_{Fy}} -s \, \ad_y \circ F \bigg] (T) \bigg)\hphantom{.}}\\
		\quad& = & \bigg(  \bigg[  e^{t \ad_{u(y)}} e^{s \, \ad_{Fy}} \bigg](S) + \bigg[\frac{1 - e^{t \ad_{u(y)}}}{\ad_{u(y)}} H|_y e^{s \, \ad_{Fy}} -s  F \bigg](T), \bigg[ e^{s \, \ad_{Fy}} -s \, \ad_y \circ F \bigg] (T) \bigg).&\hfill\qedhere
	\end{IEEEeqnarray*}
\end{proof}

{\bf Acknowledgements:} The authors would like to thank T. Baier, J. Hilgert and O. Kaya for discussions.
The authors were partially supported by the projects:\\
UID/MAT/04459/2013, PTDC/MAT-GEO/3319/2014, PTDC/MAT-OUT/28784/2017 (FCT/Portugal) and by the COST Action 
MP1405 QSPACE. MP is also supported by the Deutsche Forschungsgemeinschaft fellowship DFG ZH 605/1-1.

\end{document}